\documentclass[a4paper,12pt]{amsart}
\usepackage{amsmath}
\usepackage{amsfonts}
\usepackage{amssymb}
\usepackage{fourier}
\usepackage{tikz}
\usepackage{float}
\usepackage[a4paper,left=2cm,right=2cm,top=2cm, bottom=2cm]{geometry}
\usepackage{lineno}
\newtheorem{lemm}{Lemma}%[section]
\newenvironment{lemma}{\begin{lemm}\stepcounter{rotcount}}{\end{lemm}}
\newtheorem{theo}[lemm]{Theorem}
\newenvironment{theorem}{\begin{theo}\stepcounter{rotcount}}{\end{theo}}
\newtheorem{coro}[lemm]{Corollary}
\newenvironment{corollary}{\begin{coro}\stepcounter{rotcount}}{\end{coro}}
\newtheorem{prop}[lemm]{Proposition}

\newtheorem{conj}[lemm]{Conjecture}

\newcounter{rotcount}
\newtheorem{rot}{}[rotcount]
\newenvironment{proofof}{\noindent}{\hfill$\Box$\medskip}
\newenvironment{rotproof}{\noindent}{\hfill$\lozenge$}

\newcommand{\ncont}{\nsubseteq}\newcommand{\cont}{\subseteq}
\renewcommand{\u}{\cup}\renewcommand{\i}{\cap}
\newcommand{\s}{^*}\newcommand{\del}{\backslash}
\newcommand{\emp}{\emptyset}
\newcommand{\si}{{\rm si}}\newcommand{\co}{{\rm co}}

\newcommand{\lc}{\left\lceil}\newcommand{\rc}{\right\rceil}
\newcommand{\defin}{\textbf}
\newcommand{\C}{\mathcal{C}}
\newcommand{\F}{\mathcal{F}}
\newcommand{\X}{\mathcal{X}}

\newcommand{\A}{\mathcal{A}}
\newcommand{\B}{\mathcal{B}}

\title{Contractible edges in $3$-connected graphs that preserve a minor}
\author{Jo\~ao Paulo Costalonga}\thanks{The author was partially supported by CNPq, grant 478053/2013-4}
\address{{\upshape joaocostalonga@gmail.com}\\
  Universidade Federal do Esp\'irito Santo\\
  Av. Fernando Ferrari, 514; Campus de Goiabeiras\\
  29075-910 - Vit\'oria - ES - Brazil
  	 }

\begin{document}

\begin{abstract}
Let $G$ be a $3$-connected graph with a $3$-connected (or sufficiently small) simple minor $H$. We establish that $G$ has a forest $F$ with at least $\left\lceil(|G|-|H|+1)/2\right\rceil$ edges such that $G/e$ is $3$-connected with an $H$-minor for each $e\in E(F)$. Moreover, we may pick $F$ with $|G|-|H|$ edges provided $G$ is triangle-free. These results are sharp. Our result generalizes a previous one by Ando et. al., which establishes that a $3$-connected graph $G$ has at least $\left\lceil|G|/2\right\rceil$ contractible edges. As another consequence, each triangle-free $3$-connected graph has an spanning tree of contractible edges. Our results follow from a more general theorem on graph minors, a splitter theorem, which is also established here. 
\end{abstract}
\maketitle


\section{Introduction}
\label{sec-intro}
The graphs we consider are allowed to have loops and parallel edges.  A graph $G$ is said to be \defin{$k$-connected} if the remotion of each set of vertices of $G$ with less than $k$ vertices leaves a connected graph (we do not consider the usual requirement that $|G|>k$). An edge $e$ of a $3$-connected graph $G$ is said to be \defin{contractible} if $G/e$ is $3$-connected. We refer the reader to \cite{Kriesell} for more about contractible edges. The following result will be generalized here. 

\begin{theorem}\label{ando}(Ando, Enomoto and Saito~\cite{Ando}) Every $3$-connected graph $G$ has at least $\lc|G|/2\rc$ contractible edges.
\end{theorem}

If $G$ is a $3$-connected graph with a simple $H$-minor (a minor isomorphic to $H$), we say that $e$ is an \defin{$H$-contractible} edge of $G$ if $G/e$ is $3$-connected with an $H$-minor. We establish:

\begin{theorem}\label{main2}
Let $G$ be a $3$-connected graph with a $3$-connected simple minor $H$. Then $G$ has a forest with $\lc(|G|-|H|+1)/2\rc$ $H$-contractible edges.
\end{theorem}

Theorem \ref{main2} for $H\cong K_1$ implies Theorem \ref{ando}, with the additional thesis that the $\lc|G|/2\rc$ contractible edges are in a forest. An interesting consequence of Theorem \ref{main2} is:

\begin{corollary}
Let $G$ be a $3$-connected graph with a $3$-connected simple minor $H$ and a subgraph $K$. Then $G$ has a forest $F$ with $\lc(|G|-|H|+1)/2\rc-|K|+1$ edges avoiding $E(K)$, such that $G/e$ is $3$-connected with an $H$-minor and having $K$ as subgraph for each $e\in F$ (considering that the labels of $V(K)$ are kept in $G/e$).
\end{corollary}

Whittle~\cite{Whittle} established the particular case that $|G|-|H|\le 2$ in Theorem \ref{main2} (more generally for matroids). When $|G|-|H|=3$, we have the following strengthening:

\begin{corollary}\label{costa-cor}(Costalonga~\cite[Corollary 1.8]{Costalonga2})
Suppose that $G$ is a $3$-connected graph with a $3$-connected simple minor $H$ and $|G|-|H|\ge 3$. Then $G$ has a forest with $3$ $H$-contractible edges.
\end{corollary}

Corollary \ref{costa-cor} also holds for matroids (Theorem 1.3 of \cite{Costalonga2}). When $G$ has no triangles, we may improve Theorem \ref{main2}:

\begin{theorem}\label{main3}
Suppose that $G$ is a triangle-free $3$-connected graph with a $3$-connected simple minor $H$. Then $G$ has a forest with $|G|-|H|$ edges which are $H$-contractible.
\end{theorem}

Although Egawa et. al.~\cite{Egawa} proved that a sufficiently large $3$-connected graph $G$ has $|G|+5$ contractible edges, the number of $H$-contractible edges in Theorem \ref{main3} is sharp. We conjecture that the analogue of Theorem \ref{main3} also holds for matroids, what is not true for Theorem \ref{main2}, because $M:=M\s(K_{3,n}''')$ has only $3$-elements $e$ such that $\si(M/e)$ is $3$-connected, see \cite[Theorem 2.10]{Wu}. Theorem \ref{main3} also yields the following corollary:

\begin{corollary}
If $G$ is a triangle-free $3$-connected graph, then $G$ has a spanning forest whose edges are contractible.
\end{corollary}

In order to prove Theorems \ref{main2} and \ref{main3}, we establish a more general results, but, first, we will need some definitions. We define an \defin{wye} of $G$ as a subgraph of $G$ isomorphic to the star with $3$ edges. We say that a simple subgraph $F$ of $G$ is a \defin{fan} of $G$ if:
\begin{itemize}
 \item [(F1)] $F$ has at least $3$ edges,
 \item [(F2)] $E(F)$ has an ordering $a_0,a_1,\dots,a_{m+1}$ of its distinct edges such that, for $i=1,\dots n$, $\{a_{i-1},a_i,a_{i+1}\}$ induces a wye or a triangle in $G$ and,
 \item [(F3)] for $0<i<m$, $\{a_{i-1},a_i,a_{i+1}\}$ induces a wye in $G$ if and only if $\{a_{i},a_{i+1},a_{i+2}\}$ induces a triangle. 
\end{itemize}

\begin{figure}
\hfill
\begin{minipage}{7cm}\centering\caption{ }A triangle-to-triangle fan.\label{pic-triangle}
\begin{tikzpicture}[scale=1.2]
\tikz[label distance=0mm];
\draw (20:2.5cm) -- (50:2.5cm) -- (80:2.5cm);
\draw (100:2.5cm) -- (130:2.5cm) -- (160:2.5cm);
\draw (0,0) -- (20:2.5cm) -- (0,0)  -- (0,0) -- (50:2.5cm) -- (0,0) -- (80:2.5cm);
\draw (0,0) -- (100:2.5cm) -- (0,0) -- (130:2.5cm) -- (0,0) -- (160:2.5cm);
\node (dots1)  at  (90:2.49cm)  {$\dots$};
\node (u) at (0,0) [label=below: $u$, label distance=0.1cm] {};
\node (v1) at (160:2.8cm) {$v_1$};\node (v2) at (130:2.8cm) {$v_2$};\node (v3) at (100:2.7cm) {$v_3$};\node (vn-2) at (80:2.7cm){$v_{n-2}$};\node (vn-1) at (47:2.8cm){$v_{n-1}$};\node (vn) at (20:2.8cm) {$v_n$};
\node (x1) at (145:2.6cm) {$x_1$};\node (x2) at (115:2.6cm) {$x_2$};\node (xn-2) at (63:2.6cm){$x_{n-2}$};\node (xn-1) at (32:2.8cm){$x_{n-1}$};
\node (y1) at (167:1.6cm) {$y_1$};\node (y2) at (138:1.6cm) {$y_2$};\node (y3) at (108:1.6cm) {$y_3$};\node (vn-1) at (65:1.7cm){$y_{n-2}$};\node (vn-1) at (37:1.6cm){$y_{n-1}$};\node (vn) at (13:1.6cm) {$y_n$};
\fill (160:2.5cm) circle (0.07cm);\fill  (130:2.5cm) circle (0.07cm);\fill  (100:2.5cm) circle (0.07cm);\fill  (80:2.5cm) circle (0.07cm);\fill  (50:2.5cm) circle (0.07cm);\fill (20:2.5cm) circle (0.07cm);\fill (0,0) circle(0.07cm);
\end{tikzpicture}
\end{minipage}\hfill\begin{minipage}{7cm}\centering\caption{ }\label{pic-both}A wye-to-triangle fan.
\begin{tikzpicture}[scale=1.2]
\tikz[label distance=0mm];
\draw (20:2.5cm) -- (50:2.5cm) -- (80:2.5cm);
\draw (100:2.5cm) -- (130:2.5cm) -- (160:2.5cm) -- (180:2.5cm);
\draw (0,0) -- (20:2.5cm) -- (0,0)  -- (0,0) -- (50:2.5cm) -- (0,0) -- (80:2.5cm);
\draw (0,0) -- (100:2.5cm) -- (0,0) -- (130:2.5cm) -- (0,0) -- (160:2.5cm);
\node (dots1)  at  (90:2.49cm)  {$\dots$};
\node (u) at (0,0) [label=below: $u$, label distance=0.1cm] {};
\node (v0) at (180:2.8cm) {$v_0$};\node (v1) at (160:2.8cm) {$v_1$};\node (v2) at (130:2.8cm) {$v_2$};\node (v3) at (100:2.7cm) {$v_3$};\node (vn-2) at (80:2.7cm){$v_{n-2}$};\node (vn-1) at (47:2.8cm){$v_{n-1}$};\node (vn) at (20:2.8cm) {$v_n$};
\node (x0) at (170:2.7cm) {$x_0$};\node (x1) at (145:2.6cm) {$x_1$};\node (x2) at (115:2.6cm) {$x_2$};\node (xn-2) at (63:2.6cm){$x_{n-2}$};\node (xn-1) at (32:2.8cm){$x_{n-1}$};
\node (y1) at (167:1.6cm) {$y_1$};\node (y2) at (138:1.6cm) {$y_2$};\node (y3) at (108:1.6cm) {$y_3$};\node (vn-1) at (65:1.7cm){$y_{n-2}$};\node (vn-1) at (37:1.6cm){$y_{n-1}$};\node (vn) at (13:1.6cm) {$y_n$};
\fill (180:2.5cm) circle (0.07cm);\fill (160:2.5cm) circle (0.07cm);\fill  (130:2.5cm) circle (0.07cm);\fill  (100:2.5cm) circle (0.07cm);\fill  (80:2.5cm) circle (0.07cm);\fill  (50:2.5cm) circle (0.07cm);\fill (20:2.5cm) circle (0.07cm);\fill (0,0) circle(0.07cm);
\end{tikzpicture}
\end{minipage}\hfill
\vspace{1cm}
\begin{minipage}{7cm}\centering\caption{ }A wye-to-wye fan.\label{pic-wye}
\begin{tikzpicture}[scale=1.2]
\tikz[label distance=0mm];
\draw (0:2.5cm) -- (20:2.5cm) -- (50:2.5cm) -- (80:2.5cm);
\draw (100:2.5cm) -- (130:2.5cm) -- (160:2.5cm) -- (180:2.5cm);
\draw (0,0) -- (20:2.5cm) -- (0,0)  -- (0,0) -- (50:2.5cm) -- (0,0) -- (80:2.5cm);
\draw (0,0) -- (100:2.5cm) -- (0,0) -- (130:2.5cm) -- (0,0) -- (160:2.5cm);
\node (dots1)  at  (90:2.49cm)  {$\dots$};
\node (u) at (0,0) [label=below: $u$, label distance=0.1cm] {};
\node (v0) at (180:2.8cm) {$v_0$};\node (v1) at (160:2.8cm) {$v_1$};\node (v2) at (130:2.8cm) {$v_2$};\node (v3) at (100:2.7cm) {$v_3$};\node (vn-2) at (80:2.7cm){$v_{n-2}$};\node (vn-1) at (47:2.8cm){$v_{n-1}$};\node (vn) at (20:2.8cm) {$v_n$};\node (vn+1) at (0:2.9cm) {$v_{n+1}$};
\node (x0) at (170:2.7cm) {$x_0$};\node (x1) at (145:2.6cm) {$x_1$};\node (x2) at (115:2.6cm) {$x_2$};\node (xn-2) at (63:2.6cm){$x_{n-2}$};\node (xn-1) at (32:2.8cm){$x_{n-1}$};\node (xn) at (10:2.7cm)  {$x_n$};
\node (y1) at (167:1.6cm) {$y_1$};\node (y2) at (138:1.6cm) {$y_2$};\node (y3) at (108:1.6cm) {$y_3$};\node (vn-1) at (65:1.7cm){$y_{n-2}$};\node (vn-1) at (37:1.6cm){$y_{n-1}$};\node (vn) at (13:1.6cm) {$y_n$};
\fill (180:2.5cm) circle (0.07cm);\fill (160:2.5cm) circle (0.07cm);\fill  (130:2.5cm) circle (0.07cm);\fill  (100:2.5cm) circle (0.07cm);\fill  (80:2.5cm) circle (0.07cm);\fill  (50:2.5cm) circle (0.07cm);\fill (20:2.5cm) circle (0.07cm);\fill (0:2.5cm) circle (0.07cm);\fill (0,0) circle(0.07cm);
\end{tikzpicture}
\end{minipage}
\end{figure}

In this case, we say that $a_0,a_1,\dots,a_{m+1}$ is a \defin{fan ordering} of $F$. It is easy to check that a fan must be isomorphic to one of the graphs in Figures \ref{pic-triangle}, \ref{pic-both} or \ref{pic-wye}, where $u,v_0,\dots,v_{n+1}$ are pairwise distinct with the possible exception that $v_0$ and $v_n$ may be equal in figure \ref{pic-both} and $v_0$ and $v_{n+1}$ may be equal in figure \ref{pic-wye}. Note that, if $G$ is $3$-connected and $v_0=v_n$ or $v_{n+1}$, then $G$ is  a wheel.  To simplify our language, when there is no risk of confusion, we may identify a fan of $G$ with its edge-set or with one of its fan orderings. We say that a fan is \defin{triangle-to-triangle}, \defin{wye-to-triangle} or \defin{wye-to-wye}, according whether they begin or end with triangles or wyes, as described in Figures \ref{pic-triangle}, \ref{pic-both} and \ref{pic-wye}. The edges $y_1,\dots,y_n$ like in the figures are the \defin{spokes} of $F$, the vertex $u$ is the \defin{hub} of $F$ and the path induced by the edges other than the spokes is called the \defin{rim} of $F$. 

Suppose that $F^+$ is a maximal wye-to-wye fan of $G$ (this is, $F^+$ is not a proper subgraph of other wye-to-wye fan of $G$). Let $x_0,y_1,x_1\dots,y_n,x_n$ be a fan ordering of $F^+$,  we say that $F:=G[y_1,x_1,\dots,y_n]$ is an \defin{$H$-inner fan} of $G$ provided $G/F$ is $3$-connected with an $H$-minor. An \defin{inner fan} of $G$ is an $H$-inner fan for $H\cong K_1$. An $H$-inner fan $F$ of $G$ is \defin{non-degenerated} if $|E(F)|\ge 2$. If $|E(F)|=1$, then $F$ is said to be \defin{degenerated}.

The \defin{rank} of $X\cont E(G)$ in $G$ is the number $r_G(X)$ of edges in a spanning forest of $G[X]$, or, equivalently, the number of vertices in $G[X]$ minus the number of connected components of $G[X]$. For a family $\F:=\{X_1,\dots,X_n\}$ of subsets of $E(G)$, we define the rank of $\F$ in $G$ by $r_G(\F):=r_G(X_1\u\cdots\u X_n)$ and $G[\F]:=G[X_1\u\cdots\u X_n]$; moreover, the \defin{rank-sum} of $\F$ is defined as $rs_G(\F):=r_G(X_1)+\cdots+r_G(X_n)$. We say that a family $\F$ of subsets of $E(G)$ is \defin{free} if its members are pairwise disjoint and the edge-set of each circuit of $G[\F]$ is contained in a member of $\F$. Equivalently, $\F:=\{X_1,\dots,X_n\}$ is free when $r_G(\F)=rs_G(\F)$. A family $\F$ of subsets of $E(G)$ is an \defin{$H$-fan family} if the members of $\F$ are pairwise disjoint and each member of $\F$ is an $H$-inner fan or a singleton set with an $H$-contractible edge. When we talk about an inner fan without mention to a minor $H$, it is the case that $H\cong K_1$. Now we are in conditions to state our main theorems:

\begin{theorem}\label{main-general}
Let $G$, $H'$ and $H$ be $3$-connected simple graphs such that $H$ is a minor of $H'$, $H'$ is a minor of $G$ and $|H|\ge1$. Suppose that $H'$ has a free $H$-fan family with rank $r$. Then $G$ has a free $H$-fan family with rank at least $|G|-|H'|+r$.
\end{theorem}

For $H=H'$ in Theorem \ref{main-general}, we have:

\begin{theorem}\label{main}
Let $G$ be a $3$-connected simple graph with a $3$-connected simple minor $H$ satisfying $|H|\ge1$. Then, $G$ has a free $H$-fan family with rank at least $|G|-|H|$.
\end{theorem}

For $H\cong K_1$ in Theorem \ref{main}, we may derive the following structural result:

\begin{corollary}\label{structural}
	Let $G$ be a $3$ connected graph. Then $G$ has a subgraph $F$ such that $V(F)=V(G)$ and each block of $F$ is an inner fan of $G$ or is induced by a  contractible edge of $G$.
\end{corollary}

It is clear that Theorem \ref{main3} is a corollary to Theorem \ref{main}. If $\F$ is a pairwise disjoint family of subsets of $E(G)$ and $T$ is a triangle of $G$, we say that $T$ is a \defin{crossing triangle} of $F$ if $T$ is a triangle of $G[F]$ but $E(T)$ is not contained in any member of $\F$. If we weaken the freeness condition of Theorem \ref{main-general} to the absence of crossing triangles, we have the following result:

\begin{theorem}\label{main-sum}
Let $G$, $H'$ and $H$ be $3$-connected simple graphs such that $H$ is a minor of $H'$, $H'$ is a minor of $G$ and $|H|\ge1$. Suppose that $H'$ has an $H$-fan family without crossing triangles with rank-sum $s$. Then $G$ has an $H$-fan family without crossing triangles with rank-sum at least $|G|-|H'|+s$.
\end{theorem}

For establishing Theorem \ref{main2}, it is enough to combine Theorem \ref{main} with:

\begin{theorem}\label{main-r/2}
	Let $G$ be a $3$-connected simple graph with a $3$-connected simple minor $H$. Suppose that $G$ has a free  $H$-fan family with rank $r\ge 1$. Then $G$ has a forest with $\left\lceil (r+1)/2 \right\rceil$ $H$-contractible edges.
\end{theorem}

%\begin{proposition} Let $G$ be a $3$-connected simple graph with a minor $H$. If $G$ has an $H$-fan-family with rank-sum $s$, then $G$ has at least $\left\lceil s/2 \right\rceil$ $H$-contractible edges. \end{proposition}

These results may be used to improve the bounds for the number of contractible edges in classes of graphs with fixed minors. For instance, see the next corollary, obtained from Theorems \ref{main-sum} and \ref{main-r/2} for $H:=K_1$ and $H'=K_{n,n}$.

\begin{corollary} If $G$ is a $3$-connected simple graph with an $K_{n,n}$-minor, then $G$ has a fan-family with rank sum $|G|+n^2-2n$ and $G$ has $\lc(|G|+n^2-2n+1)/2\rc$ contractible edges. Moreover $G$ has $|G|+n^2-2n$ contractible edges if $G$ is triangle-free.
\end{corollary}

All results we stated up to now  follow from Theorems \ref{main-general}, \ref{main-sum} and \ref{main-r/2}. These theorems will be proved in Section \ref{sec-proofs}.

\section{Preliminaries}
When, in a graph $G$, an edge $e$ with endvertices $u$ and $v$ is not parallel to any other edge of $G$, we say that $e=uv $ in $G$. When there is no risk of confusion, we may refer to a vertex $v$ of $G$ in a minor $H$ of $G$ as the vertex obtained from the contraction of some subgraph of $G$ containing $v$. We define the operation of vertex \defin{splitting} as the opposite of edge-contraction. We use the notation $[n]:=\{1,\dots,n\}$. We denote by $N_G(v)$ the set of neighbors of $v$ in $G$ and by $E_G(v)$ the set of edges of $G$ incident to $v$. Although some of the following lemmas are presented as corollaries to their more general versions for matroids, the reader shall have no problem to prove their graphic versions straightforwardly. 

\begin{lemma}\label{deg2}(Corollary to \cite[Proposition 8.2.7]{Oxley})
Let $G$ be a $2$-connected graph with an edge $x$ such that $G/x$ is $3$-connected but $G$ is not. Then, one of the endvertices of $x$ has exactly two neighbors in $G$.
\end{lemma}

\begin{corollary}\label{opposite vertex} 
Suppose that $T$ is a triangle in a $3$-connected simple graph $G$ such that $G/T$ is $3$-connected. Let $v\in V(T)$ and $y\in E(T)-E_G(v)$. Then $G/y$ is $3$-connected or $\deg_G(v)=3$.
\end{corollary}
\begin{proof}
Use Lemma \ref{deg2} on $G/y$ for some $x\in T-y$.
\end{proof}

We denote by $\si(G)$ the \defin{simplification} of $G$, a graph obtained from $G$ by removing all loops and deleting all but one edges in each class of parallel edges. The \defin{cosimplification} of $G$, $\co(G)$, is defined by a graph obtained from $G$ by removing all vertices with degree less than two and, in each path of $G$ maximal in respect to having all internal vertices with degree $2$, contracting all but one edges. Note that $\co(G)$ and $\si(G)$ are uniquely determined up to choosing what labels of $G$ will remain. If the reader is familiar with matroids, it is important to note that our definition of cosimplification is slightly different from that one for matroids, since we keep pairs of non-adjacent edges in a $2$-edge cut. But these definitions are coincident when $\co(G)$ is $3$-connected, which is the case we are going to use it. 

\begin{lemma}(Corollary to \cite[Lemma 3.7]{Whittle})\label{w37-g}
Suppose that $G$ is a $3$-connected graph, $T$ is a triangle and $Y$ is a wye of $G$. If $E(T)-E(Y)=\{y\}$, then $\si(G/x_1,y)\cong\si(G/x_2,y)$ for all $x_1,x_2\in E(Y)$.
\end{lemma}

\begin{lemma}(Corollary to \cite[Lemma 3.8]{Whittle})\label{w38-g}
Suppose that $G$ is a $3$-connected graph, $T$ is a triangle and $Y$ is a wye of $G$. If $E(T)-E(Y)=\{y\}$ and $E(Y)-E(T)=\{x\}$ then $G/x$ and $\co(G\del y)$ are $3$-connected.
\end{lemma}

From Lemma \ref{w38-g} we have the following corollaries:

\begin{corollary}\label{cont-biweb}
Suppose that $x_0,y_1,x_1,y_2,x_2$ is a fan ordering of a wye-to-wye fan in a $3$-connected graph $G$. Then $G/x_1$ is $3$ connected or $G$ has a wye containing $y_1$ and $y_2$.
\end{corollary}

\begin{corollary}\label{cont-triweb}
If $G$ is a $3$-connected graph with a triangle $T$ containing $3$ degree-$3$ vertices, then $G/T$ is $3$-connected. Moreover if $G$ is simple and $G\ncong K_4$, then $G/T$ is simple.
\end{corollary}

\begin{corollary}\label{cont-fan}
Suppose that $x_0,y_1,x_1,\dots,y_n,x_n$ is a fan ordering of a wye-to-wye fan of a $3$-connected simple graph $G$ with $n\ge 3$. If $1\le i \le n-1$, then $G/x_i\del y_i$ is $3$-connected and simple and has $x_0,y_1,x_1,\dots, y_{i-1},x_{i-1},y_{i+1},x_{i+1},\dots,x_n$ as the fan ordering of a wye-to-wye fan.
\end{corollary}

\begin{lemma}\label{w36-g}
Suppose that $G$ is a simple $3$-connected graph and that $x$ and $y$ are edges of $G$ such that $G/x,y$ is $3$-connected but $G/y$ is not. Then $|G|\ge 5$ and $G$ has a wye $Y$ and a triangle $T$ such that $E(T)-E(Y)=\{y\}$ and $x\in E(Y)$.
\end{lemma}
\begin{proof}
Suppose the contrary. If $|G|\le 4$, then it is clear that $G/y$ is $3$-connected. Thus $|G|\ge 5$. By Lemma \ref{deg2} on $G/y$, it follows that $x$ is incident to a vertex $u$ with exactly two neighbors $v$ and $w$ in $G/y$. Since there are no degree-$2$ vertices in $G$ and neither in $G/y$, then $u$ is incident to at least one pair $P$ of parallel edges of $G/y$. Since $G$ is simple, $G$ is obtained from $G/y$ by splitting one of the vertices incident to $P$. If $y$ is obtained by splitting $u$, then $G\del \{v,w\}$ is disconnected, a contradiction. So, $G$ is obtained by splitting one of $v$ or $w$. As $G$ is simple, $P\u y$ is the edge set of a triangle $T$ of $G$ and $Y:=G[E_G(u)]$ is a wye of $G$ meeting $P$ and containing $x$ but not $y$. This proves the lemma.
\end{proof}

\begin{lemma}\label{expansion}
Suppose that $G$ is a $3$-connected graph with $|G| \ge 4$, $e$ is an edge of $G$ other than a loop and $v$ is a vertex of $G$ not incident to $e$. Let $G'$ be the graph constructed from $G$ by putting a vertex $u$ in the middle of $e$ and adding an edge $f$ linking $u$ and $v$. Then $G'$ is $3$-connected.
\end{lemma}
\begin{proof}
Let $w$ be an endvertex of $e$ in $G$. Note that $G'/uw\cong G+vw$ is $3$-connected. If $G'$ is not $3$-connected, then, by Lemma \ref{deg2}, we have a vertex in $G'$ with only two neighbors. By construction, this implies that $G$ has a vertex with at most two neighbors. A contradiction. 
\end{proof}

\begin{lemma}\label{not-wheel}
Let $G$ be a $3$-connected graph. Suppose that $F$ is a singleton set with an edge in a wye of $G$ or $F$ is a triangle-to-triangle fan of a wye-to-wye fan of $G$. If $G/F$ is $3$-connected, then $F$ is an inner fan of $G$.
\end{lemma}
\begin{proof}
The result is clear if $G$ is a wheel. Suppose for a contradiction that $G$ is not a wheel and there is a wye-to-wye fan $F^+$ containing $F$ with $|E(F^+)|>|E(F)|+2$. In particular, we may pick $F^+$ such that  $|E(F^+)|=|E(F)|+4$. Consider the labels of $F^+$ as in Figure \ref{pic-wye}. Then $y_1$ and $y_{n-1}$ are the extreme spokes of $F$. Note that $v_n$ is a degree-$2$ vertex of $si(G/F)$. This implies that $|G/F|\le 3$ because $G/F$ is $3$-connected. If $v_0=v_{n+1}$, it is clear that $G$ is a wheel. So $v_0$, $v_n$, $v_{n+1}$ and $u$ are distinct vertices of $G/F$. A contradiction.
\end{proof}

%We use the symbol ``$\Delta$'' for the operation of symmetric difference of sets.

%\begin{lemma}Let $G$ be a graph. Suppose that $\{a,b,c\}\cont E(G)$ induces a triangle $T$ of $G$ and $C$ is a circuit of $G$ other than $T$. If $a\in C$, then $G(E(C)\Delta E(T))$ is a circuit or both $b$ and $c$ are chords of $C$.\end{lemma}
% 
\begin{lemma}\label{orthogonality}
Let $G$ be a graph, suppose that $Y,X\cont E(G)$ are sets such that $Y$ induces a wye in $G$ and $X$ is an union of edge-sets of circuits of $G$. Then $|Y\i X|\neq 1$.
\end{lemma}

\section{Lemmas}

In this section we prove some lemmas towards the proof of the theorems. We will use the symbol ``$\lozenge$'' to point the end of a nested proof. We denote by $\Pi_3$ the prism with triangular bases.

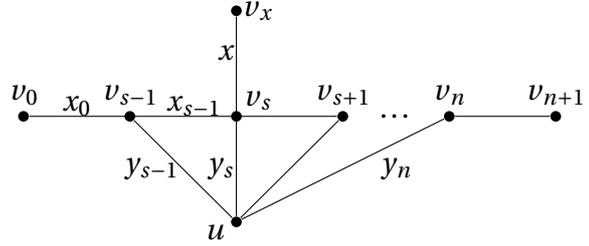
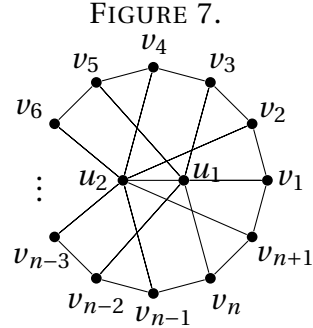
\begin{figure}[h]\begin{center}
\begin{minipage}{6cm}\centering\caption{}Graph $G_1[F\u x]$ of Lemma \ref{deletion-inner}\label{pic-deletion-g1}
\begin{tikzpicture}[scale=1.2]  \centering
\tikz[label distance=0mm];
\tikzstyle{node_style} =[shape = circle,fill = black,minimum size = 4pt,inner sep=0pt]
\node[node_style] (u) at (0,0){};
\node[node_style] (vs+1) at (0,0.9){};
\node[node_style] (vs) at (-1.0,1.5){};
\node[node_style] (vt) at (1.0,1.5){};
\node (v0) at (-2,1.5){};
\node (vn) at (2,1.5){};
\draw (v0) -- (vs) -- (vs+1) -- (vt)-- (vs) -- (u) -- (vt) --  (vn);
\draw (u)--(vs+1);
\node (x) at (0,1.65){$x$};
\node (xs) at (-1.5,1.65){$x_{s-1}$};
\node (xt+1) at (1.6,1.65){$x_{t}$};
\node (ys)  at (-0.7,0.7){$y_{s}$};
\node (yt)  at (0.7,0.7){$y_{t}$};
\node (ys+1)  at (0,0.7){$y_{s+1}$};
\node(xs+1) at (-0.3,1.33){$x_{s}$};
\node(xt) at (0.4,1.33){$x_{s+1}$};
\node () at (-2.1,1.5){$\dots$};
\node () at (2.2,1.5){$\dots$};
\node (lvs) at (-1.0,1.65){$v_s$};   
\node (lvt) at (1.0,1.65){$v_t$};
\node (lu) at (0.17,0){$u$};
\end{tikzpicture}
\end{minipage}$\qquad$
\begin{minipage}{7cm}\centering\caption{}Graph $G_2[F\u x]$ of Lemma \ref{deletion-inner}\label{pic-deletion-g2}
\begin{tikzpicture}[scale=1.4]  \centering
\tikz[label distance=0mm];
\tikzstyle{node_style} =[shape = circle,fill = black,minimum size = 4pt,inner sep=0pt]
\node[node_style] (vx) at (0,2){};\node() at (0.2,2){$v_x$};
\node[node_style] (u)  at (0,0){};\node() at (-0.2,-0.1){$u$};
\node[node_style] (vs) at (0,1){};\node() at (0.2,1.15){$v_s$};
\node[node_style] (vs+1) at (1,1){}; \node() at (1,1.2){$v_{s+1}$};
\node[node_style] (vn) at (2,1){};\node() at (2,1.2){$v_{n}$};
\node[node_style] (vn+1) at (3,1){};\node() at (3,1.2){$v_{n+1}$};
\node[node_style] (v1) at (-1,1){}; \node() at (-1,1.2) {$v_{s-1}$};
\node[node_style] (v0) at (-2,1){}; \node() at (-2,1.2) {$v_{0}$};
\node () at (1.5,1) {$\dots$};
\node () at (1.5,0.5) {$y_n$};
\node () at (-0.8,0.5) {$y_{s-1}$};
\node () at (-0.15,0.5) {$y_s$};
\node () at (-0.1,1.6) {$x$};
\node () at (-0.4,1.10) {$x_{s-1}$};
\node () at (-1.5,1.1) {$x_0$};
\draw (v0) -- (vs+1) -- (u) -- (vx);
\draw (vn+1) -- (vn) -- (u) -- (v1);
\end{tikzpicture}
\end{minipage}

\begin{minipage}{6cm}\centering\caption{}Graph $G_3[F\u x]$ of Lemma \ref{deletion-inner}\label{pic-deletion-g3}
\begin{tikzpicture}[scale=1.4]  \centering
\tikz[label distance=0mm];
\tikzstyle{node_style} =[shape = circle,fill = black,minimum size = 4pt,inner sep=0pt]
\node[node_style] (vx) at (0,2){};\node() at (0.2,2){$v_x$};
\node[node_style] (u)  at (0,0){};\node() at (-0.2,-0.1){$u$};
\node[node_style] (vs) at (0,1){};\node() at (0.15,0.85){$v_s$};
\node[node_style] (vn+1) at (1,1){}; \node() at (1.4,1){$v_{n+1}$};
\node[node_style] (v1) at (-1,1){}; \node() at (-1,1.2) {$v_{s-1}$};
\node[node_style] (v0) at (-2,1){}; \node() at (-2,1.2) {$v_{0}$};
\node () at (-0.8,0.5) {$y_{s-1}$};
\node () at (0.15,0.4) {$y_n$};
\node () at (-0.1,1.6) {$x$};
\node () at (-0.4,1.1) {$x_{s-1}$};
\node () at (0.5,1.1) {$x_{n}$};
\node () at (-1.5,1.1) {$x_0$};
\draw (v0) -- (vn+1);
\draw (v1) -- (u) -- (vx);
\end{tikzpicture}
\end{minipage}$\qquad$
\begin{minipage}{7cm}\centering\caption{}\label{pic-wheel-lift}
\begin{tikzpicture}\centering[scale=1.2]
\tikz[label distance=0mm];
\tikzstyle{node_style} =[shape = circle,fill = black,minimum size = 4pt,inner sep=0pt]
\node[node_style] (u0) at (-0.4cm,0cm){};
\node[node_style] (u1) at (0.4cm,0cm){};
\node[node_style] (v1) at (0:1.5cm){};
\node[node_style] (v2) at (30:1.5cm){};
\node[node_style] (v3) at (60:1.5cm){};
\node[node_style] (v4) at (90:1.5cm){};
\node[node_style] (v5) at (120:1.5cm){};
\node[node_style] (v6) at (150:1.5cm){};
\node (v7) at (180:1.5cm){$\vdots$};
\node[node_style] (v8) at (210:1.5cm){};
\node[node_style] (v9) at (240:1.5cm){};
\node[node_style] (v10) at (270:1.5cm){};
\node[node_style] (v11) at (300:1.5cm){};
\node[node_style] (v12) at (330:1.5cm){};
\draw (-0.4,0)--(0.4,0);
\draw (u1) -- (v1) -- (u1) -- (v3) -- (u1) -- (v5) -- (u1) -- (v9) -- (u1) -- (v11);
\draw (u0) -- (v2) -- (u0) -- (v4) -- (u0) -- (v6) -- (u0) -- (v8) --  (u0) --(v10) -- (u0) -- (v12);
\draw (v8) -- (v9) -- (v10) --(v11)--(v12)--(v1)--(v2)--(v3)--(v4)--(v5)--(v6);
\node (lu0) at (-0.8cm,0cm){$u_2$};\node (lu1) at (0.7cm,0.1cm){$u_1$};
\node(l1) at (0:1.8cm){$v_1$};
\node(l2) at (30:1.8cm){$v_2$};
\node(l3) at (60:1.8cm){$v_3$};
\node(l4) at (90:1.8cm){$v_4$};
\node(l5) at (120:1.8cm){$v_5$};
\node(l6) at (150:1.9cm){$v_6$};
\node(l8) at (215:1.8cm){$v_{n-3}$};
\node(l9) at (245:1.8cm){$v_{n-2}$};
\node(l10) at (274:1.8cm){$v_{n-1}$};
\node(l11) at (300:1.9cm){$v_n$};
\node(12) at (330:2cm){$v_{n+1}$};
\end{tikzpicture}\end{minipage}\hfill
\end{center}
\end{figure}

\begin{lemma}\label{deletion-inner}
Let $G$ be a $3$-connected simple graph with an edge $x$ such that $G\del x$ is $3$-connected with a simple minor $H$. Suppose that $F$ is a non-degenerated $H$-inner fan of $G\del x$. Then, $F$ contains the members of a free $H$-fan family of $G$ with rank $r_G(F)$.
\end{lemma}

\begin{proof}
Assume the contrary. Consider the labels for a maximal wye-to-wye fan $F^+$ of $G\del x$ containing $F$ as in Figure \ref{pic-wye}. If possible choose $F^+$ with hub having degree at least $4$ in $G$. 

If $F$ is a fan of $G$, then, as $G/F\del x$ is $3$-connected, so is $G/F$. By Lemma \ref{not-wheel}, $F$ is an $H$-inner fan of $G$ and the lemma holds. Thus, $F$ is not a fan of $G$. Hence, $x$ is incident to $v_s$ in $G$ for some $s\in [n]$. 

This implies that $\deg_G(u)=\deg_{G\del x}(u)$. If $\deg_G(u)=3$ then $F$ has all vertices with degree $3$ in $G\del x$ and we should have chosen $F^+$ with $v_s$ as hub instead of $u$. Thus $\deg_G(u)\ge 4$. By Corollary \ref{cont-biweb}, each element in the rim of $F$ is $H$-contractible in $G\del x$ and, therefore, in $G$. So, if we find an $H$-inner fan $F'$ of $G$ contained in $F$, then the family
\begin{equation*}%\label{eq-del-inner}
\big\{\{x_i\}:i\in[n-1]\text{ and }x_i\notin E(F')\big\}\u\{F'\}.
\end{equation*}
satisfies the lemma. To find $F'$ we will consider two cases.

\emph{Case 1. $x$ is incident to $v_t$ for some $t\in \{0,\dots,n+1\}-s$: }We may assume that $t>s$. As $G$ is simple, $t\ge s+2$. Define $G_1:=G/x_{s+2},\dots,x_{t-1}\del y_{s+2},\dots, y_{t-1}$ (see Figure \ref{pic-deletion-g1}). By Corollary \ref{cont-fan}, $G_1\del x$ is $3$-connected and, as a consequence, so is $G_1$. Now, note that $\{x,x_s,x_{s+1}\}$ induces a triangle and $\{y_{s+1},x_s,x_{s+1}\}$ induces a wye in $G_1$. By Lemma \ref{w38-g}, $G_1/y_{s+1}$ is $3$-connected. Now, let $F'$ be the fan of $G$ with fan ordering $y_{s+1},x_{s+1},\dots, y_{t-1}$. So, $G/F'=G_1/y_{s+1}$ is $3$-connected. By Lemma \ref{not-wheel}, $F'$ is an $H$-inner fan of $G$ and the lemma holds in Case 1.

\emph{Case 2. $v_s$ is the unique vertex of $F^+$ incident to $x$: } We may assume that $s\ge 2$. Define:
$$G_2:=G/x_1,\dots,x_{s-2}\del y_1,\dots,y_{s-2} \qquad\text{and}\qquad
G_3:=G_2/x_s,\dots,x_{n-1}\del y_{s},\dots, y_{n-1}. $$

We represent $G_2[F\u x]$ and $G_3[F\u x]$ in Figures \ref{pic-deletion-g2} and \ref{pic-deletion-g3}. We keep the labels of $v_s$ and $v_{s-1}$ in $G_2$ and $G_3$. Let $v_s$ and $v_x$ be the endvertices of $x$. By Corollary \ref{cont-fan}, $G_2\del x$ and $G_3\del x$ are $3$ connected and, therefore, so are $G_2$ and $G_3$. 

By the description of Case 2, $v_x, v_{s-1}, v_{n+1}$ and $u$ are distinct neighbors of $v_s$ in $G_3$. Thus $\deg_{G_3}(v_s)\ge 4$. Note that $G_3/\{x_{s-1},y_{s-1},y_n\}=G/F$ is $3$-connected. As $v_s$ is opposite to $y_{s-1}$ in $G_3[\{x_{s-1},y_{s-1},y_n\}]$, thus, by Corollary \ref{opposite vertex}, $G_3/y_{s-1}$ is $3$-connected.

As $G_2/y_{s-1}$ can be obtained from $G_3/y_{s-1}$ by successively applying Lemma \ref{expansion} (see Figures \ref{pic-deletion-g2} and \ref{pic-deletion-g3}), then $G_2/y_{s-1}$ is $3$-connected because so is $G_3/y_{s-1}$. Let $F'$ be the fan of $G$ with fan ordering $y_1,x_1,\dots,y_{s-1}$. Note that $G_2/y_{s-1}=G/F'$, which is $3$-connected. By Lemma \ref{not-wheel}, $F'$ is an $H$-inner fan of $G$ and the lemma holds.\end{proof}

\begin{lemma}\label{G is a wheel}
Suppose that $F$ is an inner fan of a $3$-connected graph $G$ and $|G/F|\le 3$. Then $G$ is a wheel.
\end{lemma}
\begin{proof}We may assume that $|G|\ge 5$. Consider a wye-to-wye fan $F^+$ of $G$ containing $F$ labeled as in Figure \ref{pic-wye}. If $v_{n+1}=v_0$, the result is clear, so, assume that $v_0\neq v_{n+1}$. Therefore, $V(G/F)=\{u,v_0,v_{n+1}\}$. Hence, $V(G)=V(F^+)$. As $G$ has no vertices with degree less than $3$, then $X:=\{uv_0, uv_{n+1}, v_0v_{n+1}\}\cont E(G)$. If $G$ has some edge out of $E(F^+)\u X$, we have a contradiction to the fact that $F^+$ is a wye-to-wye fan of $G$. This proves the lemma.
\end{proof}

\begin{lemma}\label{contraction-inner1}
Let $G$ be a $3$-connected graph with an edge $x$ such that $G/x$ is $3$-connected and simple with a simple minor $H$. If $F$ is a non-degenerated $H$-inner fan of $G/x$ such that $G[E(F)]$ has no triangles, then one of the spokes of $F$ is $H$-contractible in $G$.
\end{lemma}
\begin{proof}
Consider, in $G/x$, a maximal wye-to-wye fan $F^+$ containing $F$, labeled as in Figure \ref{pic-wye}. Since $G[E(F)]$ has no triangles, then $G$ is obtained from $G/x$ by splitting $u$ into two vertices $u_1$ and $u_2$ in such a way that $v_i$ is adjacent to $u_1$ in $G$ if $i$ is odd and to $u_2$ if $i$ is even. If $G/x$ is a wheel, then $G$ is isomorphic to the graph in Figure \ref{pic-wheel-lift} and, in this case, the result may be verified directly. So, assume that $G/x$ is not a wheel. By Lemma \ref{G is a wheel}, $|G/F\u x|\ge 4$. Hence, $|G|\ge 7$. Moreover, $v_0\neq v_n$.

As $G/x$ is not a wheel, then there is $u'\in N_{G/x}(u)-V(F^+)$. If $n$ is even, by symmetry, we choose the labels in such a way that $u'\in N_G(u_1)$.

We may assume that $G/v_2u_2$ is not $3$-connected. As $|G/u_2v_2|\ge 4$, thus $G/u_2v_2$ has a $2$-vertex cut. Since $G$ is $3$-connected, $G$ has a $3$-vertex cut in the form $S:=\{v_2,u_2,w\}$. Note that $w\neq u_1$ because $G/x=G/u_1u_2$ is $3$-connected. So, $G\del S$ has a vertex $s$ in a different connected component than $u_1$. Denote by $v_F$ the vertex of $G/F\u x$ obtained by the contraction of $F\u x$ in $G$ and denote $G':=G/F\u x$. 

If $s\in V(F)$, as $v_1u_1\in E(G)$, $s\neq v_1$. Thus $n\ge 3$. As $u_1v_3\in E(G)$, then $w$ is in the $(v_3,s)$-path contained in the rim of $F$. Let $v_k:=s$. As $G'$ is $3$-connected, $|G'|\ge 4$ and $v_0,v_n\in V(G')-v_F$, then $G'\del v_F$ has a $(v_n,v_0)$-path $\gamma$. Note that $v_k,v_{k+1},\dots,v_n,\gamma,v_0,v_1,u_1$ is an $(s,u_1)$-path of $G\del S$. A contradiction. Therefore, $s\notin V(F)$ and $s$ is a vertex of $G'$ distinct from $v_F$.

If $w\notin \{v_0,v_1\}$, define $\sigma:=v_0,v_1,u_1$. Otherwise, if $n\ge 3$, define $\sigma:=v_n,\dots,v_3,u_1$ and, if $n=2$, define $\sigma=u',u_1$. Denote by $t$ the first vertex of $\sigma$. So, $\sigma$ is a $(t,u_1)$-path of $G\del S$. As, $s\in V(G')-\{w,v_F\}$, $G'\del\{w,v_F\}$ has an $(s,t)$-path $\varphi$. Now, $s,\varphi,t,\sigma,u_1$ is an $(s,u_1)$-path of $G\del S$. A contradiction.
\end{proof}

\begin{lemma}\label{contraction-inner2} Let $G$ be a simple $3$-connected graph with an edge $x$ such that $G/x$ is $3$-connected and simple with a simple minor $H$. Suppose that $F$ is an $H$-inner fan of $G/x$. Consider the labels for a maximal wye-to-wye fan $F^+$ of $G$ containing $F$ as in Figure \ref{pic-wye}. If $F$ is a fan of $G$, then one of the following alternatives holds:
\begin{enumerate}
	\item [(a)] $F$ is an $H$-inner fan of $G$ or
	\item [(b)] $G$ contains an edge $y$ such that one of $x,y,x_0,y_1,x_1\dots,y_n,x_n$ or $x_0,y_1,x_1\dots,y_n,x_n,y,x$ is the fan ordering of a maximal wye-to-wye fan of $G$ containing an $H$-inner fan of $G$.
\end{enumerate}
\end{lemma}

\begin{proof} Since $F$ is a fan of $G$, then $F^+$ is a wye-to-wye fan of $G$. By Lemma \ref{not-wheel}, (a) holds if $G/F$ is $3$-connected. So, assume that $G/F$ is not $3$-connected. By Corollary \ref{cont-fan} used iteratively, $G_1:=G/x_2,\dots,x_{n-1}\del y_2,\dots,y_{n-1}$ is simple and $3$-connected. If $\deg_{G_1}(u)=3$, then $G_1/y_1,x_1,y_n=G/F$ is $3$-connected by Corollary \ref{cont-triweb}, a contradiction. Thus, $\deg_{G_1}(u)\ge 4$ and, by Corollary \ref{cont-biweb}, $G_2:=G_1/x_1\del y_1$ is $3$-connected and simple. 

Note that $G_2/x,y_n=G/F\u x$ is $3$ connected, but $G_2/y_n=G/F$ is not $3$-connected. By Lemma \ref{w36-g}, $G_2$ has a wye $Y$ and a triangle $T$ such that $E(T)-E(Y)=\{y_n\}$ and $x\in E(Y)$. But $y_n$ is in the wye induced by $\{x_0,x_n,y_n\}$ in $G_2$. So, we may assume without losing generality that $x_n\in T$ and, therefore, $E(T)=\{x_n,y_n,y\}$, where $y=uv_0$ in $G_2$. Note that $E(T)$ also induces a triangle in $G$. So $y=uv_{n+1}$ in $G$. Moreover, $x\neq y$ and, therefore, $x\in E(Y)-E(T)$. This implies that $x_0,y_1,\dots,y_n,x_n,y,x$ is a wye-to-wye fan of $G$.  To conclude (b) we have to check that $G/(E(F)\u\{x_n,y\})=G_2/x_n,y_n,y$ is $3$-connected. By Lemma \ref{w37-g},  $G_2/x,y_n\cong\si(G_2/y,y_n)\cong \si(G_2/x_n,y_n,y)$. As  $G_2/x,y_n=G/F\u x$, then $G_2/x_n,y_n,y$ is $3$-connected and the lemma is valid.
\end{proof}

\begin{lemma}\label{deg3 hub}
Let $G$ be a $3$-connected simple graph with $|G|\ge 4$ and with an edge $x$ such that $G/x$ is $3$-connected and simple. Suppose that $F$ is an inner fan of $G/x$ and $G$ is obtained from $G/x$ by splitting the hub of $F$. Consider the labels of a wye-to-wye fan $F^+$ of $G/x$ containing $F$ as in figure \ref{pic-wye}. If, for $k\in [n-1]$, $G/x_k$ is not $3$-connected, then $x_{k-1},y_k,x_k,y_{k+1},x_{k+1}$ is the fan ordering of a maximal wye-to-wye fan of $G$ with a degree-$3$ hub.
\end{lemma}
\begin{proof}Suppose the contrary. Since $G$ is obtained for $G/x$ by splitting the hub $u$ of $F$, then $\deg_{G/x}(u)\ge 4$. This implies that $|G/x|\ge 5$. By Corollary \ref{cont-biweb}, $G/x,x_k$ is $3$-connected. By Lemma \ref{w36-g}, there is a wye $Y$ of $G$ meeting a triangle $T$ such that $x\in E(Y)$ and $E(T)-E(Y)=\{x_k\}$. As $G/x$ is $3$-connected and simple, $x$ is in no triangle of $G$, and, therefore, $T$ is a triangle of $G/x$. As $G/x$ is $3$-connected with $|G/x|\ge 5$ and $x_k$ is in the rim of a fan of $G$, then it is straightforward to verify that $T$ is the unique triangle of $G/x$ containing $x_k$. Then, $E(T)=\{x_k,y_k,y_{k+1}\}$. As a consequence, $Y=\{x,y_k,y_{k+1}\}$ and the lemma holds.
\end{proof}

\begin{lemma}\label{contraction-inner}Let $G$ be a $3$-connected simple graph, other than a wheel, not isomorphic to to $\Pi_3$ and with an edge $x$ such that $G/x$ is $3$-connected and simple with a simple minor $H$. Suppose that $F$ is a non-degenerated $H$-inner fan of $G/x$ but $G[E(F)]$ is not a fan of $G$. Then $E(F)$ contains a free $H$-fan family $\X$ of $G$ such that:
	\begin{enumerate}
	    \item[(a)] $r_G(\X)=r_{G/x}(F)$,
	    \item[(b)] $\X\u\{\{x\}\}$ is a free-family of $G$ and
	    \item[(c)] one of the members of $\X$ contains an edge incident to the hub of $F$ in $G/x$ and the other members are singleton sets in the rim of $F$.
	\end{enumerate}
Moreover, $G[E(F)]$ contains at most one triangle $T$ with three degree-$3$ vertices and when such triangle exists it is a member of $\X$.
\end{lemma}

\begin{figure}\centering
\begin{minipage}{9.5cm}\centering\caption{}\label{pic-cont-1} 
Labeling of $G$ in Lemma \ref{contraction-inner} ($m=5$).
\begin{tikzpicture}\centering[scale=1.4]
\tikz[label distance=0mm];
\tikzstyle{node_style} = [shape = circle,fill = black,minimum size = 2pt,inner sep=2pt]
\tikzstyle{node_style2} = [shape = circle,fill = black,minimum size = 0.5pt,inner sep=0.5pt]
\fill[gray!20,opacity=0.3] (162:4cm) -- (-1,0) -- (146:4cm)-- (150:4cm)--(154:4cm)--(158:4cm)-- cycle;
\fill[gray!20,opacity=0.3] (130:4cm) -- (1,0) -- (114:4cm)--(118:4cm) --(122:4cm)--(126:4cm) --cycle;
\fill[gray!20,opacity=0.3] (98:4cm) -- (-1,0) -- (82:4cm) --(86:4cm) --(90:4cm)-- (94:4cm)-- cycle;
\fill[gray!20,opacity=0.3] (66:4cm) -- (1,0) -- (50:4cm) --(54:4cm) --(58:4cm)--(62:4cm)--cycle;
\fill[gray!20,opacity=0.3] (34:4cm) -- (-1,0) -- (18:4cm) -- (22:4cm)--(26:4cm)--(30:4cm)--cycle;
\node[node_style] (vn+1) at (2:4cm) {};	\node[node_style] (t5) at (18:4cm) {};		\node[node_style2] (5v3) at (22:4cm) {};	\node[node_style2] (5v2) at (26:4cm) {};	\node[node_style2] (5v1) at (30:4cm) {};
\node[node_style] (s5) at (34:4cm) {};	\node[node_style] (t4) at (50:4cm) {};		\node[node_style2] (4v3) at (54:4cm) {};	\node[node_style2] (4v2) at (58:4cm) {};	\node[node_style2] (4v1) at (62:4cm) {};
\node[node_style] (s4) at (66:4cm) {};	\node[node_style] (t3) at (82:4cm) {};		\node[node_style2] (3v3) at (86:4cm) {};	\node[node_style2] (3v2) at (90:4cm) {};	\node[node_style2] (3v1) at (94:4cm) {};
\node[node_style] (s3) at (98:4cm) {};	\node[node_style] (t2) at (114:4cm) {};		\node[node_style2] (2v3) at (118:4cm) {};	\node[node_style2] (2v2) at (122:4cm) {};	\node[node_style2] (2v1) at (126:4cm) {};
\node[node_style] (s2) at (130:4cm) {};	\node[node_style] (t1) at (146:4cm) {};		\node[node_style2] (1v3) at (150:4cm) {};	\node[node_style2] (1v2) at (154:4cm) {};	\node[node_style2] (1v1) at (158:4cm) {};
\node[node_style] (s1) at (162:4cm) {};	\node[node_style] (v0) at (178:4cm) {};
\node[node_style] (u1) at (-1,0) {};	\node[node_style] (u2) at (1,0) {};
\draw(u1)--(u2);	\draw(v0)--(s1);	\draw(t1)--(s2);	\draw(t2)--(s3);	\draw(t3)--(s4);	\draw(t4)--(s5);	\draw(t5)--(vn+1);	\draw (s1) -- (u1) -- (t1);	\draw (s2) -- (u2) -- (t2);	\draw (s3) -- (u1) -- (t3);	\draw (s4) -- (u2) -- (t4);	\draw (s5) -- (u1) -- (t5);
\draw[style=help lines] (s1)--(1v1)--(u1)--(1v2)--(u1)--(1v3)--(t1)--(1v3)--(1v2)--(1v1);
\draw[style=help lines] (s2)--(2v1)--(u2)--(2v2)--(u2)--(2v3)--(t2)--(2v3)--(2v2)--(2v1);
\draw[style=help lines] (s3)--(3v1)--(u1)--(3v2)--(u1)--(3v3)--(t3)--(3v3)--(3v2)--(3v1);
\draw[style=help lines] (s4)--(4v1)--(u2)--(4v2)--(u2)--(4v3)--(t4)--(4v3)--(4v2)--(4v1);
\draw[style=help lines] (s5)--(5v1)--(u1)--(5v2)--(u1)--(5v3)--(t5)--(5v3)--(5v2)--(5v1);
\node(lvn+1) at (2:4.5cm) {$v_{n+1}$};	\node(lt5) at (18:4.35cm) {$v_{t_5}$};	\node(ls5) at (34:4.35cm) {$v_{s_5}$};	\node(lt4) at (50:4.35cm) {$v_{t_4}$};	\node(ls4) at (66:4.35cm) {$v_{s_4}$};	\node(lt3) at (82:4.35cm) {$v_{t_3}$};	\node(ls3) at (98:4.35cm) {$v_{s_3}$};	\node(lt2) at (114:4.35cm) {$v_{t_2}$};	\node(ls2) at (130:4.35cm) {$v_{s_2}$};	\node(lt1) at (146:4.35cm) {$v_{t_1}$};	\node(ls1) at (162:4.35cm) {$v_{s_1}$};	\node(lv0) at (178:4.35cm) {$v_0$};	\node (lu1) at (-1,-0.3) {$u_1$};	\node (lu2) at (1,-0.3) {$u_2$};
\node at (-3,1.4) {$F_1$};	\node at (-1.7,3) {$F_2$};	\node at (-0.2,3.2) {$F_3$};	\node at (1.9,2.7) {$F_4$};	\node at (3,1.5) {$F_5$};
\end{tikzpicture}\end{minipage}
\begin{minipage}{6.5cm}\centering\caption{}\label{pic-cont-2} 
Graph $G_1$ of Lemma \ref{contraction-inner} ($m=5$)
\begin{tikzpicture}[scale=0.9]
\tikz[label distance=0mm];
\tikzstyle{node_style} = [shape = circle,fill = black,minimum size = 2pt,inner sep=2pt]
\tikzstyle{node_style2} = [shape = circle,fill = black,minimum size = 0.5pt,inner sep=0.5pt]
\node[node_style] (vn+1) at (0:3cm) {};	\node[node_style] (t5) at (30:3cm) {};	\node[node_style] (t4) at (60:3cm) {};	\node[node_style] (t3) at (90:3cm) {};	\node[node_style] (t2) at (120:3cm) {};	\node[node_style] (t1) at (150:3cm) {};	\node[node_style] (v0) at (180:3cm) {};	\node[node_style] (u1) at (-0.7,0) {};	\node[node_style] (u2) at (0.7,0) {};	
\node (lvn+1) at (0:3.6cm) {$v_{n+1}$};	\node (lt5) at (30:3.3cm) {$v_{t_5}$};	\node (lt4) at (60:3.3cm) {$v_{t_4}$};	\node (lt3) at (90:3.3cm) {$v_{t_3}$};	\node (lt2) at (120:3.4cm) {$v_{t_2}$};	\node	(lt1) at (150:3.4cm) {$v_{t_1}$};	\node (lv0) at (180:3.4cm) {$v_{0}$};	\node (lu1) at (-0.7,-0.3) {$u_1$};	\node (lu2) at (0.7,-0.3) {$u_2$};
\draw(u1)--(u2);	\draw (v0)--(t1)--(t2)--(t3)--(t4)--(t5)--(vn+1);	\draw (t1)--(u1)--(t3)--(u1)--(t5);	\draw (t2)--(u2)--(t4);	\node (lt5) at (25:2.2cm) {$y_{t_5}$};	\node (lt4) at (50:2.3cm) {$y_{t_4}$};	\node (lt3) at (88:2.1cm) {$y_{t_3}$};	\node (lt2) at (122:2.1cm) {$y_{t_2}$};	\node (lt1) at (159:2cm) {$y_{t_1}$};
\end{tikzpicture}\end{minipage}\hfill
\begin{minipage}{6.5cm}\centering\caption{}\label{pic-cont3} 
Graph $G_2$ of Lemma \ref{contraction-inner}\\ ($m=5$, $\alpha=3$)
\begin{tikzpicture}[scale=0.9]
\tikz[label distance=0mm];
\tikzstyle{node_style} = [shape = circle,fill = black,minimum size = 2pt,inner sep=2pt]
\tikzstyle{node_style2} = [shape = circle,fill = black,minimum size = 0.5pt,inner sep=0.5pt]
\node[node_style] (vn+1) at (0:3cm) {};	\node[node_style] (t5) at (30:3cm) {};	\node[node_style] (t4) at (60:3cm) {};	\node[node_style] (t2) at (120:3cm) {};	\node[node_style] (t1) at (150:3cm) {};	\node[node_style] (v0) at (180:3cm) {};	\node[node_style] (u1) at (-0.7,0) {};	\node[node_style] (u2) at (0.7,0) {};
\node (lvn+1) at (0:3.6cm) {$v_{n+1}$};	\node (lt5) at (30:3.3cm) {$v_{s_5}$};	\node (lt4) at (60:3.3cm) {$v_{s_4}$};	\node (lt2) at (120:3.4cm) {$v_{t_2}$};	\node (lt1) at (150:3.4cm) {$v_{t_1}$};	\node (lv0) at (180:3.4cm) {$v_{0}$};	\node (lu1) at (-0.7,-0.3) {$u_1$};	\node (lu2) at (0.7,-0.3) {$u_2$};
\draw(u1)--(u2);	\draw (v0)--(t1)--(t2)--(u1)--(t1)--(u1)--(t4);	\draw (t5)--(u1)--(t2)--(u2)--(t4)--(t5)--(vn+1);
\node (lt5) at (25:2.2cm) {$y_{s_5}$};
\node (lt4) at (49:2.3cm) {$y_{s_4}$};
%\node (lt3) at (88:2.1cm) {$y_{t_3}$};
\node (lt2) at (105:1.8cm) {$y_{t_2}$};
\node (lt1) at (159:2cm) {$y_{t_1}$};
\end{tikzpicture}\end{minipage}\hfill
\begin{minipage}{10cm}\centering\caption{}\label{pic-cont-4} 
Graph $(G/F_\alpha)$ of Lemma \ref{contraction-inner} ($m=5$, $\alpha=2$)
\begin{tikzpicture}\centering[scale=1.3]
\tikz[label distance=0mm];
\tikzstyle{node_style} = [shape = circle,fill = black,minimum size = 2pt,inner sep=2pt]
\tikzstyle{node_style2} = [shape = circle,fill = black,minimum size = 0.5pt,inner sep=0.5pt]
\fill[gray!20,opacity=0.3] (162:4cm) -- (-1,0) -- (146:4cm)-- (150:4cm)--(154:4cm)--(158:4cm)-- cycle;
\fill[gray!20,opacity=0.3] (130:4cm) -- (1,0) -- (114:4cm)--(118:4cm) --(122:4cm)--(126:4cm) --cycle;
\fill[gray!20,opacity=0.3] (66:4cm) -- (1,0) -- (50:4cm) --(54:4cm) --(58:4cm)--(62:4cm)--cycle;
\fill[gray!20,opacity=0.3] (34:4cm) -- (-1,0) -- (18:4cm) -- (22:4cm)--(26:4cm)--(30:4cm)--cycle;
\node[node_style] (vn+1) at (2:4cm) {};	\node[node_style] (t5) at (18:4cm) {};	\node[node_style2] (5v3) at (22:4cm) {};	\node[node_style2] (5v2) at (26:4cm) {};	\node[node_style2] (5v1) at (30:4cm) {};
\node[node_style] (s5) at (34:4cm) {};	\node[node_style] (t4) at (50:4cm) {};	\node[node_style2] (4v3) at (54:4cm) {};	\node[node_style2] (4v2) at (58:4cm) {};	\node[node_style2] (4v1) at (62:4cm) {};
\node[node_style] (s4) at (66:4cm) {};
\node[node_style] (t2) at (114:4cm) {};		\node[node_style2] (2v3) at (118:4cm) {};	\node[node_style2] (2v2) at (122:4cm) {};	\node[node_style2] (2v1) at (126:4cm) {};
\node[node_style] (s2) at (130:4cm) {};	\node[node_style] (t1) at (146:4cm) {};	\node[node_style2] (1v3) at (150:4cm) {};	\node[node_style2] (1v2) at (154:4cm) {};	\node[node_style2] (1v1) at (158:4cm) {};
\node[node_style] (s1) at (162:4cm) {};	\node[node_style] (v0) at (178:4cm) {};
\node[node_style] (u1) at (-1,0) {};	\node[node_style] (u2) at (1,0) {};
\draw(u1)--(u2);	\draw(v0)--(s1);	\draw(t1)--(s2);	\draw(t2)--(u1);	\draw(u1)--(s4);	\draw(t4)--(s5);	\draw(t5)--(vn+1);	\draw (s1) -- (u1) -- (t1);	\draw (s2) -- (u2) -- (t2);	\draw (s4) -- (u2) -- (t4);	\draw (s5) -- (u1) -- (t5);
\draw[style=help lines] (s1)--(1v1)--(u1)--(1v2)--(u1)--(1v3)--(t1)--(1v3)--(1v2)--(1v1);
\draw[style=help lines] (s2)--(2v1)--(u2)--(2v2)--(u2)--(2v3)--(t2)--(2v3)--(2v2)--(2v1);
\draw[style=help lines] (s4)--(4v1)--(u2)--(4v2)--(u2)--(4v3)--(t4)--(4v3)--(4v2)--(4v1);
\draw[style=help lines] (s5)--(5v1)--(u1)--(5v2)--(u1)--(5v3)--(t5)--(5v3)--(5v2)--(5v1);
\node(lvn+1) at (2:4.5cm) {$v_{n+1}$};	\node(lt5) at (18:4.35cm) {$v_{t_5}$};	\node(ls5) at (34:4.35cm) {$v_{s_5}$};	\node(lt4) at (50:4.35cm) {$v_{t_4}$};	\node(ls4) at (66:4.35cm) {$v_{s_4}$};	\node(lt2) at (114:4.35cm) {$v_{t_2}$};	\node(ls2) at (130:4.35cm) {$v_{s_2}$};	\node(lt1) at (146:4.35cm) {$v_{t_1}$};	\node(ls1) at (162:4.35cm) {$v_{s_1}$};	\node(lv0) at (178:4.35cm) {$v_0$};	\node (lu1) at (-1,-0.3) {$u_1$};	\node (lu2) at (1,-0.3) {$u_2$};	\node at (-3,1.4) {$F_1$};	\node at (-0.7,1.8) {$F_2$};	\node at (1.9,2.7) {$F_4$};	 \node at (3,1.5) {$F_5$}; 
\end{tikzpicture}\end{minipage}\hfill
\end{figure}
\begin{proof}
	Consider a maximal wye-to-wye fan $F^+$ of $G/x$ containing $F$, labeled as in Figure \ref{pic-wye}. Since $G[E(F)]$ is not a fan of $G$, then $G$ is obtained from $G/x$ by splitting $u$ into vertices $u_1$ and $u_2$.  Let $F_1,\dots,F_m$ be the maximal subsets of $E(F)$ such that each $G[F_k]$ is a triangle-to-triangle fan of $G$ or $F_k$ is a singleton set with a spoke of $F$. Let $y_{s_k}$ and $y_{t_k}$ be the extreme spokes of $F_k$ with $s_k\le t_k$, which are incident to $v_{s_k}$ and $v_{t_k}$, respectively. Choose the labels in such a way that $k>l$ implies $s_k>s_l$ (this labeling is illustrated in Figure \ref{pic-cont-1}). First we check:
	
	\begin{rot}\label{t0}
	There is at most one index $k\in [m]$ such that $F_k$ is a triangle of $G$ with $3$ degree-$3$ vertices.
	\end{rot}
	\begin{rotproof}
	Suppose the contrary. Let $1\le i<j\le m$ be such indices. Say that $u_1$ is a vertex of $F_i$. So $E_G(u_1)=\{y_{s_i},y_{t_i},x\}$. Thus $u_2\in V(F_j)$. Analogously, $E_G(u_2)=\{y_{s_j},y_{t_j},x\}$. Thus $y_{s_i},y_{t_i},y_{s_j}$ and $y_{t_j}$ are the unique spokes of $F$ and $n=3$. Define $W:=\{u_1,u_2,v_1,v_2,v_3,v_4\}$. If $G$ has a vertex $v\in V(G)-W$, then it is clear that $v$ and $u_1$ are in different connected components of $G\del\{v_1,v_4\}$. Thus $V(G)=W$. Now it is straightforward to check that $G\cong\Pi_3$ or to $\mathcal{W}_4$. A contradiction to the hypothesis.
	\end{rotproof}

	\begin{rot}\label{t1}
	For some $\alpha\in[m]$, $G/F_\alpha$ is $3$-connected and each edge $z$ in the rim of $F$ and out of $E(F_\alpha)$ is $H$-contractible in $G$. 
	\end{rot}
	\begin{rotproof}We consider two cases for this.
	
	\emph{Case 1: For some $\alpha\in [m]$, $F_\alpha$ is a triangle with $3$ degree-$3$ vertices: } By Lemma \ref{cont-triweb}, $G/F_\alpha$ is $3$-connected. By \ref{t0}, each edge $z$ in the rim of $F$ and out of $F_k$ is not in a triangle with $3$ degree-$3$ vertices and, therefore, by Lemma \ref{deg3 hub}, $z$ is $H$-contractible in $G$. Note that the second part of the lemma is proved.
	
	\emph{Case 2: Otherwise: } By Lemma \ref{deg3 hub}, each edge in the rim of $G$ is $H$-contractible in $G$. We just have to find $\alpha\in [m]$ such that $G/F_\alpha$ is $3$-connected. For $k\in [m]$, define $Y_k:=\{y_{s_k},\dots, y_{t_k-1}\}$ and $X_k:=\{x_{s_k},\dots, x_{t_k-1}\}$. Moreover, let: $$G_1:=G/X_1\u\cdots\u X_m\del Y_1\u\cdots\u Y_m.$$
	
	Note that the unique edge of $G_1$ remaining from each $F_i$ is $y_{t_i}$ (see Figures \ref{pic-cont-1} and \ref{pic-cont-2}). Since $G[E(F)]$ is not a fan of $G$, then $m\ge 2$. This implies that $G_1/x$ is obtained from $G/x$ by repeatedly performing the operation of Corollary \ref{cont-fan}. Hence, $G_1/x$ is $3$-connected and simple. Now we split this case two into two subcases:
	
	\emph{Case 2.1: $G_1$ is not $3$-connected: } By Lemma \ref{deg2}, we may assume that $\deg_{G_1}(u_2)=2$ because $G_1/x$ is $3$-connected. Say that $y_{t_\alpha}$ is incident to $u_2$. Therefore $G/X_\alpha\del Y_\alpha$ has $u_2$ as a degree-$2$ vertex incident to $y_{t_\alpha}$ and $x$. Thus, $G/F_\alpha=G/X_\alpha\del Y_\alpha/y_{t_\alpha}\cong (G/x)/X_\alpha\del Y_\alpha$, which is $3$-connected by Corollary \ref{cont-fan}. So, we have the desired $\alpha$ in this case.%	Hence, $G_2:=G/x_{s_\alpha+1}\dots,x_{t_\alpha-1}\del y_{s_\alpha+1}\dots,y_{t_\alpha-1}$ has $T:=\{y_{s_\alpha},x_{s_\alpha},y_{y_k}\}$ as a triangle with $3$ degree-$3$ vertices. By Corollary \ref{cont-fan}, $G_2$ is $3$-connected and simple and, by Lemma \ref{cont-triweb}, so is $G_2/T$. But $G_2/T=G/F_\alpha$. So $G/F_\alpha$ is $3$-connected.
	
	\emph{Case 2.2: $G_1$ is $3$-connected: } Now, $F'_+:=x_ 0,y_{t_1},x_{t_1},y_{t_2}, \dots, x_{t_{m-1}},y_{t_m},x_n$ is a maximal wye-to-wye fan of $G_1/x$. Since $m\ge 2$, then $F'=y_{t_1},x_{t_1},y_{t_2}, \dots, x_{t_{m-1}},y_{t_m}$ is a maximal triangle-to-triangle fan of $G_1/x$ contained in $F'_+$ (see Figure \ref{pic-cont-2}). Since $G_1/F'\u x=G/F\u x$ is $3$-connected, then $F'$ is an $H$-inner fan of $G_1/x$ by Lemma \ref{not-wheel}.
	By construction, none of the edge-sets of triangles of $F'$ is the edge-set of a triangle of $G_1$, thus, by Lemma \ref{contraction-inner1}, for some $\alpha\in[m]$, $y_{t_\alpha}$ is $H$-contractible in $G_1$. Say that $y_{t_\alpha}$ is incident to $u_1$. Consider a graph $G_2$ obtained from $G_1/y_{t_\alpha}$ by changing the label of $y_{t_i}$ by $y_{s_i}$ for $i>\alpha$ (see Figure \ref{pic-cont3}). Now $G/F_\alpha$ can be rebuilt from $G_2$ using the operation described in Lemma \ref{expansion} (see Figures \ref{pic-cont3} and \ref{pic-cont-4}). Therefore, $G/F_\alpha$ is $3$-connected and \ref{t1} holds.\end{rotproof}
	
	%For $i\neq t_\alpha, s_\alpha-1$, we have that $x_i$ is not incident to $u_1$ in $G_3$. Note that $G/F_\alpha$ may be obtained from $G_3$ by putting back the edges of $X_k$ and $Y_k$ for $k\neq\alpha$ using the operation of Lemma \ref{expansion} using $u=u_1$ and $e=x_i$ with $i\neq x_{t_\alpha}, x_{s_\alpha-1}$. This implies that $G/F_\alpha$ is $3$-connect. So \ref{t1} holds.\end{rotproof}
	Now, by Lemma \ref{not-wheel}, $F_\alpha$ is an $H$-inner fan of $G$. Let $$\X:=\big\{\{x_i\}:i\in[n-1]\text{ and }x_i\notin F_\alpha\big\}\u \{F_\alpha\}.$$
	Recall that, for $k\in [m]$, $x_k$ is $H$-contractible in $G$ if $x_k\notin F_\alpha$. Now, items (a), (b) and (c) are easy to verify.
\end{proof}

\section{Proofs for the Theorems}\label{sec-proofs}

In this section we prove theorems \ref{main-general}, \ref{main-sum} and \ref{main-r/2}. We define the vertex \defin{cleaving} operation as the inverse of the identification of non-adjacent vertices. The next lemma is the key for proving Theorems \ref{main-sum} and \ref{main-r/2}.

\begin{lemma}\label{resolve-tudo}
	Let $G$ and $H$ be $3$-connected simple graphs such that $H$ is a minor of $G$. Suppose that $G$ is not isomorphic to $\Pi_3$ and neither to a wheel. Suppose that $x$ is an edge of $G$ such that some $G'\in \{G/x,G\del x\}$ is $3$-connected with an $H$-minor and $\F$ is an $H$-fan family of $G'$ without crossing triangles. Then $G$ has an $H$-fan family $\X$ such that:
	\begin{enumerate}
		\item [(a)] $\X$ has no crossing triangles,
		\item [(b)] $\X$ is free if $\F$ is free and
		\item [(c)] $rs_G(\X)\ge rs_{G/x}(\F)+1$ if $G'=G/x$ and $rs_G(\X)\ge rs_{G}(\F)$ if $G'=G\del x$ .
	\end{enumerate}
\end{lemma}
\begin{proof} Write $\F:=\{F_1,\dots,F_m\}$. We first make the proof in the simple case, when $G'=G\del x$. If $F_k$ is an $H$-inner fan of $G\del x$, then, by Lemma \ref{deletion-inner}, $F_k$ contains a free $H$-fan family $\F_k$ with rank $r_G(F_k)$. Otherwise, $F_k$ is singleton and contains an $H$-contractible element of $G\del x$ and, therefore, of $G$. In this case we define $\F_k:=\{F_k\}$. It is straightforward to check that $\X:=\F_1\u\cdots\u \F_m$ is the family we are looking for in this case. 

Now, assume that $G'=G/x$. Let $F$ be the union of the members of $\F$. Next, we define a partition $\{I_1,I_2,I_3,J_1,J_2,J_3,K,L\}$ of $[m]$ and families $\X_k$, for $k\in [m]-L$ as follows. First we will define the sets $I_1$, $I_2$ and $I_3$. For $i\in [m]$, we let:
\begin{itemize}
	\item $i\in I_1$ if $G/F_i$ is not $3$-connected and $F_i$ is a fan of $G$,
	\item $i\in I_2$ if $G/F_i$ is not $3$-connected, $|F_i|=1$ and $F_i$ is in a wye of $G$ and
	\item $i\in I_3$ if $G/F_i$ is not $3$-connected, $|F_i|=1$ and $F_i$ is not in a wye of $G$.
\end{itemize}
For $i\in I_1$, let $F^+$ be a wye-to-wye fan of $G$ containing $F_i$ with $|E(F^+_i)|-|E(F_i)|=2$. By Lemma \ref{contraction-inner2}, $G$ has an edge $\psi(i)$ such that, for some ordering $x^i_0,y^i_1,x^i_1,\dots,y^i_{n_i},x^i_{n_i}$ of $F^+_i$ and for $\chi(i):=x^i_0$, we have that $x,\psi(i),\chi(i),y^i_1,x^i_1,\dots,y^i_{n_i},x^i_{n_i}$ is a wye-to-wye fan of $G$ and $F'_i:=\psi(i),\chi(i),y^i_1,x^i_1,\dots,y^i_{n_i}$ is an $H$-inner fan of $G$. In this case, we define $\X_i:=\{F'_i\}$. Note that $G/\psi(i)$ is not $3$-connected for $i\in I_1$.

For $i\in I_2\u I_3$, we denote $F_i=\{y^i_1\}$. By Lemma \ref{w36-g}, as $x$ is in no triangle of $G$, there are edges $\chi(i)$ and $\psi(i)$ such that $\{x,\chi(i),\psi(i)\}$ induces a wye and $F'_i:=\{\chi(i),\psi(i), y^i_1\}$ induces a triangle of $G$. By Lemma \ref{w37-g}, $\si(G/F'_i)\cong \si(G/F_i\u x)$ is $3$-connected with an $H$-minor. 

If $i\in I_2$, then, by Lemma \ref{not-wheel}, $F'_i$ is an $H$-inner fan of $G$ and we define $\X_i:=\{F'_i\}$. Moreover, we pick the labels of $\chi(i)$ and $\psi(i)$ in such a way that $\chi(i)$ is in a wye of $G$ with $y^i_1$. In particular, this implies that $G/\psi(i)$ is not $3$-connected. 

If $i\in I_3$, then by, Corollary \ref{opposite vertex}, $\chi(i)$ and $\psi(i)$ are $H$-contractible in $G$. In this case, we define $\X_k:=\big\{\{\chi(k)\},\{\psi(k)\}\big\}$. We may pick the labels of $\chi(i)$ and $\psi(i)$ in such a way that $\psi(i)\notin F$ because of the following:

\begin{rot}\label{new2}
If $i\in  I$, then $|\{\chi(i),\psi(i)\}\i F|\le 1$
\end{rot}
\begin{rotproof}
If $\chi(i),\psi(i)\notin F_i$, then $F'_i$ is a crossing triangle of $\F$. This proves \ref{new2}.
\end{rotproof}

Moreover, as we observed before:
\begin{rot}\label{new1}
If $i\in I_1\u I_2$, then $G/\psi(i)$ is not $3$-connected.
\end{rot}

Define $I=I_1\u I_2\u I_3$. We defined functions $\psi,\chi:I\rightarrow E(G)$. For each $j\in [m]-I$ such that $F_j$ do not intersect $\psi(I)\u\chi(I)$, we let:

\begin{itemize}
\item $j\in J_1$ if $|F_j|=1$ and $G/F_j$ is $3$-connected,
\item $j\in J_2$ if $|F_j|>1$, $F_j$ is a fan of $G$ and $G/F_j$ is $3$-connected and
\item $j\in J_3$ if $|F_j|>1$ and $F_j$ is not a fan of $G$.
\end{itemize}

For $j\in J_1$, we simply define $\X_k:=\{F_j\}$.

If $j\in J_2$, then by Lemma \ref{not-wheel}, $F_j$ is an $H$-inner fan of $G$ and we define $\X_j:=\{F_j\}$.

For $j\in  J_3$, by Lemma \ref{contraction-inner}, there is a free $H$-fan family $\X_j$ of $G$ with the members contained in $F_j$ satisfying items (a), (b) and (c) of such lemma.

We define $J:=J_1\u J_2\u J_3$. For each $k\in [m]-I$ with $F_k$ meeting $\psi(I)\u\chi(I)$, we let:

\begin{itemize}
	\item $k\in K$ if $|F_k|>1$ and
	\item $k\in L$ if $|F_k|=1$.
\end{itemize}

For $k\in K$, we define $\X_k$ as the partition of the edge-set of the rim of $F_k$ in singleton sets. We will check on \ref{new4} (viii) that $\X_k$ is a free $H$-fan family of $G$.

We will not define $\X_k$ for $k\in L$. Observe that $\{I_1,I_2,I_3,J_1,J_2,J_3,K,L\}$ is indeed a partition of $[m]$. Moreover each $\X_i$ is a free $H$-fan family of $G$. Next we prove:

\begin{rot}\label{new6}
If $\{i,j\}$ is a $2$-subset of $I$, then $\{\chi(i),\psi(i)\}\i F'_j=\{\chi(i),\psi(i)\}\i(\{\chi(j),\psi(j)\}\u F_j)=\emp$.
\end{rot}
\begin{rotproof}
First we check that $\{\chi(i),\psi(i)\}\i\{\chi(j),\psi(j)\}=\emp$. Suppose the contrary. Then, for $k=i,j$, $Y_k:=\{\chi(k),\psi(k),x\}$ induces a wye of $G$. But this implies that $Y_i=Y_j$ since such wyes have a common pair of edges. Thus, $\{\chi(i),\psi(i)\}=\{\chi(j),\psi(j)\}$. Moreover, for $k=i,j$, $T_k:=G[\{\chi(k),\psi(k),y^k_1\}]$ is a triangle. As $G$ is simple, $y^i_1=y^j_1$ and, therefore, $F_i$ intersects $F_j$. A contradiction.

Now it is left to show that $\{\chi(i),\psi(i)\}\i F_j=\emp$. Suppose for a contradiction that  $z\in \{\chi(i),\psi(i)\}\i F_j$. As $G(F'_j)$ is a union of circuits of $G$, by Lemma \ref{orthogonality}, $Y_i$ meets $F'_j$ in at least two edges. As $x\notin F'_j$, then $\chi(i),\psi(i)\in F'_j$. As $\{\chi(i),\psi(i)\}\i\{\chi(j),\psi(j)\}=\emp$, thus $\{\chi(i),\psi(i)\}\cont F'_j-\{\chi(j),\psi(j)\}\cont F_j\cont F$. A contradiction to \ref{new2}.
\end{rotproof}

By \ref{new6}, for each $l\in L$ there is an unique index $\varphi(l)\in I$ such that $F_l$ is either equal to $\{\chi(\varphi(l))\}$ or $\{\psi(\varphi(l))\}$. This defines a function $\varphi:L\rightarrow I$. By \ref{new2}, $\varphi$ is injective. We will extend the domain of $\varphi$ to $L\u K$ further. Next we prove:

\begin{rot}\label{new4}
If $k\in K$, then there is an unique index $i\in I$ such that $\{\chi(i),\psi(i)\}$ meets $F_k$. Moreover:
\begin{itemize}
\item [(i)] $F_k$ is not a fan of $G$. In particular, $G$ is obtained from $G/x$ by splitting the hub of $F_k$. 
\item [(ii)] $i\in I_2$.
\item [(iii)] $\chi(i)$ is a spoke of $F_k$.%may say extreme spoke
\item [(iv)] $\psi(i)\notin F$.
\item [(v)] $\psi(I)\u \chi(I)$ meets no member of $\X_k$,%
\item [(vi)] $i\notin \varphi(L)$ and $|I|>|L|$.
\item [(vii)] $|K|=1$.
\item [(viii)] $\X_k$ is a free $H$-fan family of $G$.
\end{itemize}
\end{rot}
\begin{rotproof} By the definition of $K$, for some $i\in I$ there is an element $z\in \{\chi(i),\psi(i)\}\i F_k$. 

To prove (i), suppose for a contradiction that $F_k$ is a fan of $G$. Thus $F_k$ is an union of circuits of $G$. But $Y:=G[\{\chi(i),\psi(i),x\}]$ is a wye of $G$ meeting $F_k$ and, by Lemma \ref{orthogonality}, $Y$ meets $F_k$ in at least two edges. By \ref{new2}, $x\in F_k$, a contradiction. Thus, $F_k$ is not a fan of $G$. The second part of (i) follows straightforwardly from this fact. So, (i) holds.

Say that $G$ is obtained from $G/x$ by splitting the hub of $F_k$ into vertices $u_1$ and $u_2$ linked by $x$. Since $z$ is adjacent to $x$, then $z$ is a spoke of $F_k$ and we may assume that $z$ is incident to $u_1$. Let $v_1$ be the other endvertex of $z$ in $G$. Since $v_1$ is in the rim of $F_k$, then $E_G(v_1)$ induces a wye $Y_1$ of $G$ meeting the triangle induced by $F'_i:=\{\chi(i),\psi(i),y^i_1\}$. As $x$ and $v_1$ are not incident, then $x\notin E(Y_1)$ and $Y_1\neq Y$. So $Y$ and $Y_1$ are distinct wyes of $G$ meeting $F'_i$ and, therefore, $F'_i\cont E(Y)\u E(Y_1)$. Since $y^i_1$ is not adjacent to $x$, hence $y^i_1\in E(Y_1)=E_G(v_1)$. So, $v_1$ is incident to an edge out of $F_k$ and, consequently, $v_1$ is an extreme of the rim of $F_k$. Let $F^+_k$ be a wye-to-wye fan of $G/x$ containing $F_k$ labeled as in Figure \ref{pic-wye}. Note that $x_0=y^i_1$ is in the wye $Y_1$ of $G$, which is also a wye of $G/x$. Then $i\notin I_3$. Moreover, as $Y_1$ meets a triangle of $F_k$, then by Lemma \ref{w38-g}, $y^i_1$ is contractible in $G/x$. This implies that $i\notin I_1$. So, $i\in I_2$ and (ii) holds.

By Lemma \ref{w37-g}, $\si(G/x,y^i_1)\cong\si(G/F'_i)$ is $3$-connected. If $\deg_G(v_0)=3$, then $\deg_{G/x}(v_0)=3$ and we have a contradiction to the maximality of $F^+_k$ as a wye-to-wye fan of $G/x$. Thus $\deg_G(v_0)\neq 3$. By Lemma \ref{opposite vertex}, $G/z$ is $3$-connected. By (ii) and \ref{new1}, $z=\chi(i)$ and (iii) holds.

Note that (iv) follows directly from \ref{new2}. We checked that if $j\in I$ and $z'\in \{\psi(j),\chi(j)\}\i F_k$, then $z'=\chi(j)$ is a spoke of $F_k$ and $\psi(j)\notin F$. This implies (v) because the members of $\X_k$ are in the rim of $F_k$. By (iv), $\{\psi(i)\}\notin \F$. Moreover, $\{\chi(i)\}\notin \F$. Hence, there is no index $l\in L$ for which $i=\varphi(l)$. This implies that $\varphi$ is not surjective and (vi) holds.

Now we check that $\deg_G(u_2) \ge 4$. Since $E_G(u_1)=E(Y)$, $\deg_G(u_1)=3$. Moreover, $G[\{v_0,u_1,v_1\}]$ is a triangle. Suppose for a contradiction that $\deg_G(u_2)=3$. Hence, $N_G(u_2)=\{u_1,v_2,v_3\}$ and, for $X:=\{u_1,u_2,v_0,v_1,v_3,v_4\}$, $G[X]$ has as subgraph the graph in figure \ref{pic-new4}, where $v_1$, $v_2$, $u_1$ and $u_2$ have degree $3$ in $G$. As $G\del v_0,v_3$ is connected, then $V(G)=X$ and $G\cong \Pi_3$ or $\mathcal{W}_4$, a contradiction the hypothesis. So, $\deg_G(u_2) \ge 4$.

For proving (vii), suppose for a contradiction that $j\in K-k$. Note that $u_1$ is a degree-$3$ endvertex of $x$ in $G$ incident to $\psi(i)\notin F$ and to $\chi(i)\in F_k$. Analogously, for $j$, one endvertex $u$ of $x$ has degree $3$ and is incident to an edge of $F_j$ and an edge out of $F\u x$. Clearly, $u=u_2$. But this contradicts the fact that $\deg_G(u_2) \ge 4$.

For proving (viii) it is enough to check that each edge in the rim of $F$ is $H$-contractible in $G$. As $\deg_G(u_2) \ge 4$, it follows from Lemma \ref{deg3 hub}.

It is left to prove the uniqueness of $i$. Suppose for a contradiction that, for some $j\in I-i$. As $E_G(u_1)=\{x,\psi(i),\chi(i)\}$, analogously, for $j$, one endvertex $w$ of $x$ satisfies $N_G(w)=\{x,\psi(j),\chi(j)\}$. By \ref{new6}, $u=u_2$. But this contradicts the fact that $\deg_G(u_2) \ge 4$.
\end{rotproof}

\begin{figure}
\begin{minipage}{7cm}\centering\caption{}\label{pic-new4}
\begin{tikzpicture}[scale=2]  \centering
\tikz[label distance=0mm];
\tikzstyle{node_style} =[shape = circle,fill = black,minimum size = 4pt,inner sep=0pt]
\node[node_style] (v0) at (-1,0.5){};
\node[node_style] (v1) at (0,1){};
\node[node_style] (v2) at (1,1){};
\node[node_style] (v3) at (2,0.5){};
\node[node_style] (u1) at (0,0){};
\node[node_style] (u2) at (1,0){};
\draw (u1)--(v0)--(v1)--(v2)--(v3)--(u2);\draw(v1)--(u1)--(u2)--(v2);
\node (lv0) at (-1,0.6){$v_0$};
\node (lv1) at (0,1.1){$v_1$};
\node (lv2) at (1,1.1){$v_2$};
\node (lv3) at (2,0.6){$v_3$};
\node (lu1) at (0,-0.15){$u_1$};
\node (lu2) at (1,-0.15){$u_2$};
\node	at (0.5,-0.07){$x$};
\node at (0.17,0.5){$\chi(i)$};
%\node at (0.83,0.5){$\chi(j)$};
\node at (-0.5,0.4){$\psi(i)$};
%\node at (1.5,0.4){$\psi(j)$};
\end{tikzpicture}
\end{minipage}\end{figure}

By items (vi), (vii) and (vii) of \ref{new4}, we may extend the function $\varphi$ previously defined:
\begin{rot}\label{new5}
There is an injective function $\Phi:K\u L\rightarrow I$ such that:
\begin{itemize}
 \item If $k\in K$, then $\Phi(k)$ is the unique index $i\in I$ such that $\chi(i)\in F_k$.
 \item If $l\in L$, $\Phi(l):=\varphi(l)$ is the index $i\in I$ such that $\{\chi(i),\psi(i)\}$ meets $F_l$.
\end{itemize}
\end{rot}
By \ref{new5}, $|I|\ge|K|+|L|$, then, in every possible case, we may define $\X$ as follows:
\begin{equation}\label{eq-defx}
\X:=\begin{cases}
\big\{\{x\}\big\}\u\left(\bigcup\limits_{k\in [m]-L}\X_k\right) & \text{ if } |I|=|K|+|L|\medskip\\%K\neq\emp \text{ or }|L|=|I|\medskip\\
\bigcup\limits_{k\in [m]-L}\X_k  & \text{ if } |I|>|K|+|L|.
\end{cases}
\end{equation}
We will prove that $\X$ is a family satisfying the lemma. Denote by $X$ the union of the members of $\X$. We shall prove now:
\begin{rot}\label{new7}
The members of $\X$ are pairwise disjoint.
\end{rot}
\begin{rotproof}Suppose for a contradiction that there are distinct members $A$ and $B$ in $\X$ with a common element $z$. By construction, each family $\X_k$ has pairwise disjoint members and does not contain $\{x\}$. 
So, there are distinct $i,j\in[m]$ such that $A\in \X_i$ and $B\in \X_j$. Note that each member of $\X_k$ is contained in $F_k\u \chi(I)\u\psi(I)$ for $k\in [m]-L$. Hence, if $z\notin \chi(I)\u\psi(I)$, then $z\in F_i\i F_j$, contradicting the disjointness of $\F$. We may assume that $i\in I$ and $z\in \{\chi(i),\psi(i)\}$. In particular $z$ is in the wye $Y:=G[\{\chi(i),\psi(i),x\}]$.

By \ref{new6}, $j\notin I$. If $j\in J$, then, by definition, $F_j$ does not meet $\chi(I)\u\psi(I)$. So, $z\notin F_j$. Therefore, $z$ is in no member of $\X_j$ by construction. Hence, $j\notin J$. The remaining possibility is that $j\in K$. but this contradicts \ref{new4} (v).
\end{rotproof}
\begin{rot}\label{new8}
$rs_G(\X)\ge rs_{G/x}(\F)+1$
\end{rot}
\begin{rotproof}
	By \ref{new5}, $|I|-|K|-|L|\ge 0$ and, by \eqref{eq-defx}, $|X\i\{x\}|+|I|-|K|-|L|\ge 1$. Moreover, observe that $rs_G(\X_i)=r_G(\X_i)=r_{G/x}(F_i)+1$ for each $i\in I$, $rs_G(\X_j)=r_G(\X_j)=r_{G/x}(F_j)$ for each $j\in J$, $rs_G(\X_k)=r_G(\X_k)=r_{G/x}(F_k)-1$ for each $k\in K$ and $r_G(F_l)=r_{G/x}(F_l)=1$ for each $l\in L$. Therefore, the rank-sum of $\X$ is given by:
	\begin{equation*}
	\begin{array}{rcl}
	rs_G(\X)&=&
	|X\i\{x\}|+\sum\limits_{k\in [m]-L}rs_{G}(\X_k)
	\\&=&
	|X\i\{x\}|+\sum\limits_{i\in I}(r_{G/x}(F_i)+1)+\sum\limits_{j\in J}r_{G/x}(F_j)+\sum\limits_{k\in K}(r_{G/x}(F_k)-1)
	\\&=&
	|X\i\{x\}|+|I|-|K|-|L|+\sum\limits_{k\in [m]}r_{G/x}(F_k)
	\\&=&
	|X\i\{x\}|+|I|-|K|-|L|+rs_{G/x}(\F)
	\\&\ge&
	rs_{G/x}(\F)+1.
	\end{array}
	\end{equation*}
	This proves \ref{new8}.
\end{rotproof}

\begin{rot}\label{new15}
Suppose that $k\in [m]$ and $D$ is a circuit of $G$ with $E(D)\cont (F_k\u x)$. Then either
\begin{itemize}
 \item  $k\notin K$ and $E(D)$ is contained in a member of $\X$ or
 \item  $k\in K$ and $E(D)\ncont X$.
\end{itemize}

\end{rot}
\begin{rotproof}
	Since $|F_k|\ge|D|-1\ge 2$, then $k\in I_1\u J_2\u J_3\u K$. 

	If $k\in I_1\u J_2$, then $F_k$ is a fan of $G$. So, it is clear that $G[F_k\u x]$ has no circuits containing $x$, and therefore $E(D)\cont F_k$, which is contained in a member of $\X_k$. So, we may assume that $k\in J_3\u K$.

	If $k\in J_3$, then $\X_k$ satisfies item (c) of Lemma \ref{contraction-inner} and \ref{new15} holds. 

	So, assume that $k\in K$. Let $i$ be the index given by \ref{new4}. Note that $D$ has at least two spokes of $F_k$. Let $s$ be a spoke of $F_k$ in $D$ other that $\chi(i)$. Suppose for a contradiction that $s\in X$. Then $s$ is in a member of $\X_j$ for some $j\in [m]-L$. By the definition of $\X_k$, $j\neq k$. By the uniqueness of $i$ and by \ref{new4} (iv), $s\notin \chi(I)\u \psi(I)$. But, by construction, $E(G[\X_j])\cont F_j\u \chi(I)\u\psi(I)$. So, $s\in F_k\i F_j$, contradicting the disjointness of $\F$.
\end{rotproof}

\begin{rot}\label{new20}
	$\X$ has no crossing triangles.
\end{rot}
\begin{rotproof}
	Suppose for a contradiction that $T$ is a crossing triangle of $\X$. As $G/x$ is simple, $x\notin T$ and $T$ is a triangle of $G/x$. 
	
	If $E(T)\cont F$, then, as $\F$ has no crossing triangles, $T\cont F_k$ for some $k\in [m]$ and, by \ref{new15}, $E(T)$ is contained in a member of $\X$ or $E(T)\ncont X$, a contradiction. 
	
	Thus $E(T)\ncont F$ and there is an edge $z\in E(T)\i \{\chi(i),\psi(i)\}$ for some $i\in I$. Recall that $\{x,\psi(i),\chi(i)\}$ induces a wye in $G$. As $x\notin T$, hence $\{\chi(i),\psi(i)\}\cont T$ and, therefore, $E(T)=\{\chi(i),\psi(i),y^i_1\}$. If $i\in I_1\u I_2$, then $T$ is contained in $F'_i$, but $\X_i=\{F'_i\}$ in this case, a contradiction. So, $i\in I_3$. But, now, $y^i_1\in T-X$ by construction. A contradiction again. So, \ref{new20} holds.
\end{rotproof}

Items (a) and (c) of the lemma follows from \ref{new7}, \ref{new8} and \ref{new20}. It is left to prove item (b), this is, it is enough to prove that $\X$ is free provide $\F$ is free to finish the proof. Suppose the contrary. By \ref{new7}, $G$ has a circuit $C$ such that $E(C)\cont X$ but $E(C)$ is contained in no member of $\X$. Choose such $C$ minimizing $|E(C)|$. We will prove some assertions next:
\begin{rot}\label{new9}
If $e\in X$, then $e$ is not a chord of $C$.
\end{rot}
\begin{rotproof}
	Suppose the contrary. Then, there are circuits $C_1$ and $C_2$ of $G$ such that $E(C)\u e=E(C_1)\u E(C_2)$ and $E(C_1)\i E(C_2)=\{e\}$. Let $A$ and $B$ be distinct members of $\X$ meeting $E(C)$ with $e\notin A$. For some $i\in[2]$, $C_i$ meets $A$ and the member of $\X$ containing $e$. Moreover, $E(C_i)\cont X$. Since $G$ is simple, $C_i$ contradicts the minimality of $C$.
\end{rotproof}
\begin{rot}\label{new10}
For each $l\in L\u K$, $\chi(\Phi(l))\in F_l\cont F$.
\end{rot}
\begin{rotproof}
	By \ref{new4} (vi), we may assume that $l\in L$. Let $i:=\Phi(l)$. So, one of $\chi(i)$ or $\psi(i)$ is in $F_l\cont F$. If $i\in I_3$, the result follows from our choice of labels for $\chi(i)$ and $\psi(i)$. Assume that $i\in I_1\u I_2$. By \ref{new1}, $G/\psi(i)$ is not $3$-connected. If $G/x,\psi(i)$ is $3$-connected, then by Lemma \ref{w36-g}, $x$ and $\psi(i)$ are not adjacent. A contradiction. As $(G/x)/F_l$ is $3$-connected, then $F_l=\{\chi(i)\}$.
\end{rotproof}

\begin{rot}$E(C)\cont F\u x$.\label{new11}
\end{rot}
\begin{rotproof}
	Define $I_\psi:=\{k\in I: \psi(k)\in E(C)\}$. Recall that, for each $k\in I$, $Y_k:=G[\{x,\psi(k),\chi(k)\}]$ is a wye and $T_k:=G[\{y^k_1,\psi(k),\chi(k)\}]$ is a triangle. Hence, $|I|\le 2$ and $|I_\psi|\le2$. We will use the symbol ``$\Delta$'' for symmetric difference of sets. Denote $I_\psi:=\{i_1,\dots,i_n\}$ and define $Z:=E(C)\Delta E(T_{i_1})\Delta\cdots\Delta E(T_{i_n})$. Note that $G[Z]$ is a union of edge disjoint circuits of $G$. By \ref{new6}, $T_1,\dots,T_n$ are pairwise disjoint. We consider two cases:

	\emph {Case (i). $x\notin C$: } In this case, for each $k\in I$, by Lemma \ref{orthogonality} applied on $C$ and $Y_k$, we have that $\{\psi(k),\chi(k)\}$ is contained in $E(C)$ or disjoint from $E(C)$. Thus $k\in I_\psi$ if and only if $\{\psi(k),\chi(k)\}\cont E(C)$. This implies that $Z\cont F$. If $I_\psi=\emp$, then $E(C)=Z$ and \ref{new11} holds, so, assume that $i\in I_\psi$. 

	If $y^i_1\in E(C)$, then $E(C)=\{y^i_1, \psi(i),\chi(i)\}$. Moreover, $y^i_1\in X$ and $i\notin I_3$. So, $i\in I_1\u I_2$. This implies that $E(C)\cont F'_i$. But $\{F'_i\}\in \X$. A contradiction. Thus $y^i_1\notin E(C)$. 

	Now, $y^i_1$ is a chord of $C$ and, by \ref{new9}, $y^i_1\notin X$. Therefore, $i\in I_3$ and $F_i:=\{y^i_1\}$. Note that $y^i_1\in Z$. Let $D$ be a circuit of $G[Z]/x$ containing $y^i_1$. As $E(D)\cont Z\cont F$ and $F_i=\{y^i_1\}\cont E(D)\ncont F_i$, thus $D$ contradicts the freeness of $\F$.

	\emph {Case (ii). $x\in C$: } Then $x\in X$. By the definition of $\X$, $|I|=|K|+|L|$. So, the function $\Phi$, defined in \ref{new5}, is surjective. Hence, by \ref{new10}, $\chi(I)\cont F$. Note that $Z\cont \chi(I)\u F\u x$, and, therefore, $Z\cont F\u x$. If $I_\psi=\emp$, then $E(C)=Z\cont F\u x$ and we have \ref{new11}. So, assume that $i\in I_\psi$. As $\psi(i)\notin Z$, then $x$ and $\chi(i)$ are incident to a common degree-$2$ vertex of $G[Z]$ and are in a same circuit $D$ of $G[Z]$. Let $k\in L\u K$ be the index such that $\chi(i)\in F_k$. As $D/x$ is a circuit of $(G/x)[F]$ and $\F$ is a free family of $G/x$, then $E(D/x)\cont F_k$. So, $|F_k|>1$ and $k\in K$. By \ref{new15}, $E(D)\ncont X$. As $E(D)\cont Z\cont X\u\chi(I)\u\{y^j_1:j\in I_\psi\}=X\u\{y^j_1:j\in I_\psi\}$, hence, for some $j\in I_\psi$, $y^j_1\in E(D)$. Thus $y^j_1\in F_k$, contradicting the disjointness of $\F$.
\end{rotproof}
\begin{rot}\label{new12}
$x$ is a chord of $C$.
\end{rot}
\begin{rotproof}
	Suppose the contrary. Thus one of $C$ or $C/x$ is a circuit of $G/x$; call such circuit $B$. By \ref{new11}, $E(B)\cont F$ and, as $\F$ is free, $E(B)\cont F_k$ for some $k\in [m]$. So, by \ref{new15}, for $D:=C$, $C\ncont X$ or $E(C)$ is in a member of $\X$. A contradiction.
\end{rotproof}

By \ref{new12}, $x$ is a chord of $C$. By \ref{new9}, $\{x\}\notin \X$. By \eqref{eq-defx}, $|I|\ge 1$ and there is $i\in I$. By \ref{new11}, $E(C)\cont F$.  Since $\{\chi(i),\psi(i),x\}$ is a wye of $G$ and $x$ is a chord of $C$, then $\{\chi(i),\psi(i)\}\cont E(C)\cont F$. But this contradicts \ref{new2}. The lemma is proved.
\end{proof}

From Seymour Splitter Theorem~\cite{Seymour1980} (we refer the reader also to \cite[Corollary 12.1.3]{Oxley}) we may conclude:

\begin{corollary}\label{seymour-splitter}
Suppose that $G$ is a $3$-connected simple graph with at least $4$ vertices and a $3$-connected simple minor $H$. If $G$ is not isomorphic to a wheel, then $G$ has and edge $x$ such that $G/x$ or $G\del x$ is $3$-connected and simple with an $H$-minor.
\end{corollary}

Now we prove Theorems \ref{main-general} and \ref{main-sum}. The same argument prove both Theorems, differing only in the use of item (b) of Lemma \ref{resolve-tudo} in the end.\smallskip

\begin{proofof}\emph{Proof of Theorems \ref{main-general} and \ref{main-sum}: }
	First, note that the theorem holds when $G$ is a wheel or $G\cong \Pi_3$ and assume the contrary. We proceed by induction on $k:=|E(G)|-|E(H')|$. When $k=0$, the result is trivial. Suppose that $k\ge 1$ and the theorem holds for smaller values of $k$. By Corollary \ref{seymour-splitter}, $G$ has an edge such that some $G'\in\{G/x,G\del x\}$ is $3$-connected and simple with an $H'$-minor. By induction hypothesis we have a fan family $\F$ of $G'$ satisfying the theorem for $G'$. By Lemma \ref{resolve-tudo}, there is a family $\X$ satisfying the theorem for $G$.
\end{proofof}

\begin{proofof}\emph{Proof of Theorem \ref{main-r/2}: }
Let $\F$ be a free $H$-fan family of $G$ with $r_G(\F)\ge r$. Consider a partition $\F=\A \u \B \u \C$, where:
\begin{itemize}
	\item $\A$ is the family of singleton sets of $\F$,
	\item $\B$ is the family of the edge sets of triangles of $G$ in $\F$ with $3$ degree-$3$ vertices and 
	\item $\C$ is the family of edge-sets of non-degenerated $H$-inner fans in $\F-\B$.
\end{itemize}
In particular, choose $\F$ maximizing $|\A|$. Let $U$ be the union of the members of $\F$. Let us check the following:
\stepcounter{rotcount}
\begin{rot}\label{r1}
If $x_0,y_1,x_1\dots,y_n,x_{n}$ is a wye-to-wye fan of $G$ containing a member $X$ of $\C$, then $\{x_0\}\in \F$ or $\{x_{n}\}\in \F$. 
\end{rot}
\begin{rotproof}
Suppose the contrary. Note that $x_0,x_n\notin U$. By Lemma \ref{w38-g}, $G/x_0$ and $G/x_n$ are $3$-connected. By Lemma \ref{w37-g}, $x_0$ and $x_{n}$ are $H$-contractible in $G$. Let $F_1$ be a spanning forest for $G[U]$. By the maximality of $|\A|$, $\F\u\big\{\{x_{n}\}\big\}$ is not free. So, $G[E(F_1)\u x_{n}]$ has a circuit $C$ containing $x_n$. Since $x_0\notin C$, then $C$ contains a spoke of $X$. Now, $F_2:=G[(E(F_1)-X)\u\{x_1,\dots,x_n\}]$ is a forest with the same number of edges as $F_1$. Note that $\F':=(\F-\{X\})\u\big\{\{x_i\}: i\in [n]\big\}$ has rank $r_G(\F)$ since it induces a subgraph of $G$ having $F_2$ as spanning forest. Moreover, each $x_i$ is $H$-contractible in $G$ by Corollary \ref{cont-biweb}. So, $\F'$ contradicts the maximality of $|\A|$.\end{rotproof}

\begin{rot}\label{r4}
If $T\in \B$ and $x_1, x_2$ and $x_3$ are the edges of $G$ adjacent to $T$, then $\left|\big\{\{x_1\},\{x_2\},\{x_3\}\big\}\i \F\right|\ge 2$.
\end{rot}
\begin{rotproof}
Note that each circuit of $G$ meeting $\{x_1,x_2,x_3\}$ also meets $E(T)$. Thus, as $\F$ is free, then so is $\F':=(\F-\{T\})\u\big\{\{x_1\},\{x_2\},\{x_3\}\big\}$. If \ref{r4} fails, then $\F'$ has rank at least $r$, contradicting the maximality of $A$.
\end{rotproof}

By \ref{r1} and \ref{r4}, $\A\neq \emp$. Let $A$ and $B$ be the union of the members of $\A$ and $\B$ respectively. We define a vertex of $G$ to be green if it is incident to an edge of $B$ and to be red otherwise. We define the non-red vertices of $G/B$ to be blue. Next, we prove:

\begin{rot}\label{r2}
$G/B$ is simple or $G$ is isomorphic to $K_4$ or $\Pi_3$.
\end{rot}
\begin{rotproof}Suppose the contrary. Let $\B:=\{B_1,\dots,B_n\}$ and let $k$ be the least index for which $G/\{B_1,\dots,B_{k}\}$ is not simple. By the second part of Corollary \ref{cont-triweb}, $G':=G/\{B_1,\dots,B_{k-1}\}\cong K_4$ and $B_k$ induces a triangle $T$ of $G'$. So, there is an unique vertex $w \in V(G')-T$. If $w$ is red, then $G=G'\cong K_4$, otherwise $w$ is blue and $G\cong \Pi_3$.
\end{rotproof}

If $G$ is isomorphic to $K_4$ or $\Pi_3$, the theorem may be verified directly. So, assume the contrary. Therefore $G/B$ is simple. 

Define $R$ as the union of the edge-sets of the rims of the members of $\C$. Moreover, define $W:=\{x\in E(G)-U: x$  is adjacent to and edge of $B\}$. Note that \ref{r4} implies:

\begin{rot}\label{r3} Each blue vertex of $G/B$ is incident to at most one edge of $W$
\end{rot}

Let $F$ be the graph obtained from $G':=(G/B)[A\u R\u W]$ by cleaving each red vertex $v$ of $G'$ into $\deg_{G'}(v)$ degree-one so said pink vertices. Note that $F$ has two types of vertices: the blue ones, with degree three, and the pink ones, with degree one.

Note that each edge of $R$ has red endvertices in $G$ and, therefore, pink endvertices in $F$. So, each edge of $R$ induces a connected component of $F$. As $\A\neq \emp$, at least one of the connected components of $F$ is not induced by an edge of $R$. Let $\kappa$ be the number of connected components of $F$. Hence:
\begin{equation}\label{consuelo}
|\C|\le |R|\le \kappa -1. 
\end{equation}

As $G/B$ is simple, $F$ is simple. As $\F$ is free, $G[A\u R]$ is a forest and so is $F[A\u R]$. Hence, $F$ has a spanning forest $T$ containing $A\u R$. Define $D:=E(F)-E(T)$. We say that a blue vertex is dark blue if it has degree $3$ in $T$; otherwise, we say that such vertex is light blue. As each edge of $D$ is in a circuit of $F$, then it has light blue endvertices. Conversely, the light blue vertices are exactly those incident to edges of $D$. As $D\cont W$, then, by \ref{r3}, each light blue vertex is incident to exactly one edge of $D$ and, therefore, have degree $2$ in $T$. Moreover, each light blue vertex has degree one in $G[D]$, and, therefore, the number $l$ of light blue vertices satisfies $l=2|D|$. Let $d$ be the number of dark blue vertices and let $T'$ be a forest obtained from $T$ by replacing by a single edge each maximal path in respect to having all inner vertices light blue. As $T'$ is a cubic forest with $\kappa$ connected components and $d$ inner vertices, then $|E(T')|=\kappa+2d$. By construction, $|E(T)|=|E(T')|+l=\kappa+2d+l$. This implies that $|E(F)|=|E(T)|+|D|=\kappa+2d+l+|D|$. As $G$, has $|\B|=l+d$ blue vertices, then $|E(F)|=\kappa+2|\B|-l+|D|$. But $l=2|D|$, so:
\begin{equation}\label{conchita}
2|\B|=|E(F)|-\kappa+|D|. 
\end{equation}

Since $D\cont W$ and $E(F)=A\dot{\u} R\dot{\u} W$, then 
\begin{equation}\label{carmemsita} 
|A|+|R|\le |E(F)|-|D|.
\end{equation}
Note that $r\le r_G(\F)= 2|\B| + (|A| + |R|) +|\C|$. So, by \eqref{conchita}, \eqref{carmemsita} and \eqref{consuelo}:
\begin{equation}\label{eq-veronica}
r\le (|E(F)|-\kappa+|D|)+(|E(F)|-|D|)+(\kappa-1)=2|E(F)|-1.
\end{equation}
This implies that $|E(F)|\ge\lc (r+1)/2 \rc$. Recall that $|E(F)|=A\u R\u W$. By Corollary \ref{cont-biweb}, the edges in $R$ are $H$-contractible in $G$. By Lemmas \ref{w38-g} and \ref{w37-g}, so are the edges in $W$. By definition, the elements of $A$ are $H$-contractible in $G$ and, therefore, so are the edges of $F$. Now, it suffices to prove that $G[E(F)]$ is a forest to establish the theorem. Indeed, recall that $E(F)=A\u R\u W$. Since $\F$ is free, then $G[A\u R]$ is a forest. So, every circuit of $G[E(F)]$ meets an edge of $W$. But each circuit meeting an edge of $W$ also meets an edge of $B$. As $E(F)\i B=\emp$, hence $G[E(F)]$ is a forest and the theorem is valid.\end{proofof}

\section{Sharpness}

We denote by $V_n(G)$ the set of vertices of $G$ with degree $n$. Consider the graphs $J_1$ and $J_2$ as in the figures below.

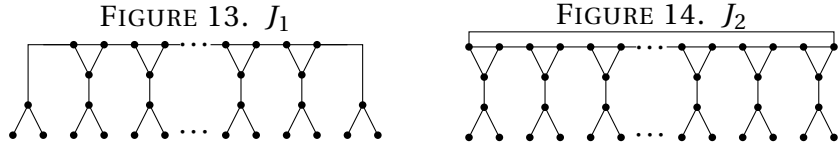
\begin{figure}[H]\centering
\begin{minipage}{6cm}\centering\caption{$J_1$}
\begin{tikzpicture}[scale=0.4]
\tikzstyle{node_style} =[shape = circle,fill = black,minimum size = 2pt,inner sep=1pt]
\node (dots) at (0,3) {$\dots$};\node (dots) at (0,0) {$\cdots$};
\node[node_style] (v1) at (1,3) {};\node[node_style] (v2) at (2,3) {};\node[node_style] (v3) at (3,3) {};\node[node_style] (v4) at (4,3) {};
\node[node_style] (v-1) at (-1,3) {};\node[node_style] (v-2) at (-2,3) {};\node[node_style] (v-3) at (-3,3) {};\node[node_style] (v-4) at (-4,3) {};
\draw (-5,3) -- (-0.5,3);\draw (5,3) -- (0.5,3);
\node[node_style] (u1) at (1.5,2) {};\node[node_style] (u2) at (3.5,2) {};\node[node_style] (u-1) at (-1.5,2) {};\node[node_style] (u-2) at (-3.5,2) {};
\node[node_style] (x1) at (1.5,1) {};\node[node_style] (x2) at (3.5,1) {};\node[node_style] (x3) at (5.5,1) {};
\node[node_style] (x-1) at (-1.5,1) {};\node[node_style] (x-2) at (-3.5,1) {};\node[node_style] (x-3) at (-5.5,1) {};
\node[node_style] (w1) at (1,0) {};\node[node_style] (w2) at (2,0) {};\node[node_style] (w3) at (3,0) {};
\node[node_style] (w4) at (4,0) {};\node[node_style] (w5) at (5,0) {};\node[node_style] (w6) at (6,0) {};
\node[node_style] (w-1) at (-1,0) {};\node[node_style] (w-2) at (-2,0) {};\node[node_style] (w-3) at (-3,0) {};
\node[node_style] (w-4) at (-4,0) {};\node[node_style] (w-5) at (-5,0) {};\node[node_style] (w-6) at (-6,0) {};
\draw[thin] (-1.5,2)--(-1.5,1);\draw[thin] (-3.5,2)--(-3.5,1);\draw[thin] (1.5,1) -- (1.5,2);\draw[thin] (3.5,1) -- (3.5,2);
\draw[thin] (5.5,1) -- (5.5,3)--(4,3);\draw[thin] (-5.5,1) -- (-5.5,3)--(-4,3);
\draw[thin] (-4,3)--(-3.5,2)--(-3,3); \draw[thin] (-2,3)--(-1.5,2)--(-1,3); 
\draw (4,3)--(3.5,2)--(3,3); \draw[thin] (2,3)--(1.5,2)--(1,3); 
\draw [thin] (-6,0)--(-5.5,1)--(-5,0); \draw [thin](-4,0)--(-3.5,1)--(-3,0); \draw [thin](-2,0)--(-1.5,1)--(-1,0); 
\draw [thin] (6,0)--(5.5,1)--(5,0);    \draw [thin](4,0)--(3.5,1)--(3,0); \draw [thin](2,0)--(1.5,1)--(1,0); 
\end{tikzpicture}\label{pic-j1}\end{minipage}\begin{minipage}{6cm}\centering\caption{$J_2$}
\begin{tikzpicture}[scale=0.4]
%\tikzstyle{node_style} = [circle,fill=black,minimum size=1pt]
\tikzstyle{node_style} =[shape = circle,fill = black,minimum size = 2pt,inner sep=1pt]
\node (dots) at (0,3) {$\dots$};\node (dots) at (0,0) {$\cdots$};
\node[node_style] (v1) at (1,3) {};\node[node_style] (v2) at (2,3) {};\node[node_style] (v3) at (3,3) {};
\node[node_style] (v4) at (4,3) {};\node[node_style] (v5) at (5,3) {};\node[node_style] (v6) at (6,3) {};
\node[node_style] (v-1) at (-1,3) {};\node[node_style] (v-2) at (-2,3) {};\node[node_style] (v-3) at (-3,3) {};
\node[node_style] (v-4) at (-4,3) {};\node[node_style] (v-5) at (-5,3) {};\node[node_style] (v-6) at (-6,3) {};
\draw (-6,3) -- (-.5,3);\draw (6,3) -- (0.5,3);
\node[node_style] (u1) at (1.5,2) {};\node[node_style] (u2) at (3.5,2) {};\node[node_style] (u3) at (5.5,2) {};
\node[node_style] (u-1) at (-1.5,2) {};\node[node_style] (u-2) at (-3.5,2) {};\node[node_style] (u-3) at (-5.5,2) {};
\node[node_style] (x1) at (1.5,1) {};\node[node_style] (x2) at (3.5,1) {};\node[node_style] (x3) at (5.5,1) {};
\node[node_style] (x-1) at (-1.5,1) {};\node[node_style] (x-2) at (-3.5,1) {};\node[node_style] (x-3) at (-5.5,1) {};
\node[node_style] (w1) at (1,0) {};\node[node_style] (w2) at (2,0) {};\node[node_style] (w3) at (3,0) {};
\node[node_style] (w4) at (4,0) {};\node[node_style] (w5) at (5,0) {};\node[node_style] (w6) at (6,0) {};
\node[node_style] (w-1) at (-1,0) {};\node[node_style] (w-2) at (-2,0) {};\node[node_style] (w-3) at (-3,0) {};
\node[node_style] (w-4) at (-4,0) {};\node[node_style] (w-5) at (-5,0) {};\node[node_style] (v-6) at (-6,0) {};
\draw[thin] (-1.5,2)--(-1.5,1);\draw[thin] (-3.5,2)--(-3.5,1);\draw[thin] (-5.5,1)--(-5.5,2);
\draw[thin] (1.5,1) -- (1.5,2);\draw[thin] (3.5,1) -- (3.5,2);\draw[thin] (5.5,1) -- (5.5,2);
\draw[thin] (-6,3)--(-5.5,2)--(-5,3); \draw[thin] (-4,3)--(-3.5,2)--(-3,3); \draw[thin] (-2,3)--(-1.5,2)--(-1,3); 
\draw[thin] (6,3)--(5.5,2)--(5,3); \draw[thin] (4,3)--(3.5,2)--(3,3); \draw[thin] (2,3)--(1.5,2)--(1,3); 
\draw[thin] (6,3)--(6,3.5)--(-6,3.5)--(-6,3);
\draw [thin] (-6,0)--(-5.5,1)--(-5,0); \draw [thin](-4,0)--(-3.5,1)--(-3,0); \draw [thin](-2,0)--(-1.5,1)--(-1,0); 
\draw [thin] (6,0)--(5.5,1)--(5,0);    \draw [thin](4,0)--(3.5,1)--(3,0); \draw [thin](2,0)--(1.5,1)--(1,0); 
\end{tikzpicture}\label{pic-j2}\end{minipage}
\end{figure}

For $i=1,2$, let $A_i$ be the set of edges in $J_i$ with some endvertex of degree one and let $B_i:=E(J_i)-A_i$. Let $2n:=|V_1(J_i)|$. For $m\ge 2n+1$, let $K_m$ be a copy of the complete graph with $m$ vertices disjoint from $J_i$. Consider the graph $G_i$ obtained by identifying $V_1(J_i)$ with $2n$ distinct vertices of $K_m$. Note that $|G[B_1]|=4n-6$ and $|G[B_2]|=4n$.

Define $H_i:=G_i/B_i$. Note that $H_i$ is $2n$-connected. Let $T_i\cont E(G_i)$ be a set such that $G_i[T_i]$ is a forest, $|G_i/T_i|=|H_i|$ and $G_i/T_i$ has an $H_i$-minor. As $G[T_i]$ is a forest, then $|T_i|=r_{G_i}(B_i)$. Choose $m\gg2n$ in such a way that $\si(G_i/x)$ has less edges than $H_i$ for each $x\in E(K_m)$. So $E(K_m)\i T_i=\emp$ and, therefore, $T_i\cont E(J_i)$. 

Let us prove that $T_i\cont B_i$. Suppose the contrary. Since $T_i\i E(K_m)=\emp$, then there is $x_1\in A_i\i T_i$. Since $m\ge2n+1$, there is a vertex $v$ in $G_i$  incident to no edges of $J_i$. Let $T_i:=\{x_1,\dots,x_r\}$. For $0\le s\le r$, define $I_s:=G_i/\{x_1,\dots, x_s\}$ and $W_s:=V(I_s)-V(I_s[E(K_m)])$. Consider the graph $J'_i$ obtained by the identification of all degree-$1$ vertices of $J_i$ into a single vertex $w_i$. Note that $|J'_i|=|G[B_i]|+1$ and $0\le s\le r=|G[B_i]|-1$. Thus $J_{i,s}:=J'_i/\{x_1,\dots,x_s\}$ has at least two vertices. Keep the label of $w_i\in J_{i,s}$. Now note that $\emp\neq V(J_{i,s})-w_i\cont W_s$. Now, observe that $I_1$ has a set with $2n-1$ vertices separating $v$ from $W_1$. By an inductive argument, we conclude that each $G_k$ has an set with less than $2n$ edges separating $v$ from $W_k\neq\emp$. So, $I_r$ is not $2n$-connected. Since $|H|=|I_r|$ and $I_r$ has an $H$-minor, then $H$ is not $2n$-connected, a contradiction. Therefore, $T_i\cont B_i$.

Since $r_G(T_i)=r_G(B_i)$, then $T_i$ induces a spanning tree of $J_i[B_i]$ and $\si(G/T_i)=G/B_i$. So, all $H_i$-contractible edges of $G_i$ are in $B_i$. Hence, for $i=1$, the largest subset of $H_1$-contractible edges of $G_1$ has $2n-3$ edges, while $|G_1|-|H_1|+1=4n-6=2(2n-3)$. Similarly, the largest subset of $H_2$-contractible edges of $G_2$ has $2n$ edges, while $|G_2|-|H_2|+1=4n=2(2n)$. This gives a sharp examples for Theorem \ref{main2} for sufficiently large odd values of $|G|-|H|$.

When $|G|-|H|$ is even, we consider for $i=1,2$, an edge $x$ in the graph $G_i$ previously defined such that $x\in B_i$ but $x$ is adjacent to an edge of $A_i$. Note that $x$ is adjacent to an unique triangle $T$, which has $3$ degree-$3$ vertices. The edge of $T$ not adjacent to $x$ is $H_i$-contractible in $G_i/x$ by Lemma \ref{cont-biweb}. Moreover, the property of the other edges of $B_i$ in being $H$-contractible in $G_i$ or $G_i/x$ is the same. As $\lc\frac{|G_i|-|H_i|+1}2\rc=\lc\frac{|G_i/x|-|H_i|+1}2\rc$, we have a sharp example for theorem \ref{main2} for $|G|-|H|$ even and sufficiently large.

For a sharp example for Theorems \ref{main3} and \ref{main}, consider two disjoint copies $G_1$ and $G_2$ of a $(k+1)$-connected triangle-free graph such that each $G_i$ has a stable $k$-set of vertices $X_i:=\{x^i_1,\dots,x^i_k\}$ and the vertices of $V(G_i)-X_i$ are not covered by less than $k+1$ edges (we may choose, for instance, $G_1$ and $G_2$ as hypercubes of a suitable size). Now consider the graph $G:=(G_1\u G_2)+\{x^1_jx^2_j: j\in[k]\}$. Define $Z:=\{x^1_jx^2_j: j\in[k]\}$ and $H:=G/Z$. Note that $H$ has an unique vertex cut $X$ with at most $k$ elements. Note that $X$ separates the edge sets of $G_1$ and $G_2$.

Now suppose that $Z'$ is a set of edges such that $H':=G/Z' \cong H$ but there is an edge $z\in Z'-Z$. Say that $z\in E(G_1)$. Now, in an inductive way, similarly as we did in the last class of examples, we may prove that there is a vertex cut $X'$ of $H'$ separating $F:=E(G_1)\i E(H')$ from $E(H')-F$. As $H\cong H'$, $X'$ is the unique vertex cut of $H'$ with up to $k$ elements. Moreover, both classes of edges separated by $X'$ induces copies of $G_1$ in $H'$. But one of then is induced by $E(G_1)\i E(H')\cont E(G_1)-z$. A contradiction. Thus $Z$ is the unique $k$-subset of $G$ such that $G/Z\cong H$. As a consequence, each $H$-contractible edge of $G$ is in $Z$.

\end{document}